\documentclass{amsart}

\usepackage{amsthm} 
\usepackage{color}
\usepackage{amsmath,amscd}

\usepackage{graphicx,epsfig}
\usepackage{epstopdf}
\usepackage{amsmath, amssymb, latexsym, euscript}
\usepackage{url}
\usepackage[all]{xy}
\usepackage{psfrag}

\def\gap{\vspace{.3cm}}

\def\smallskip{\vspace{.15cm}}
\def\medskip{\vspace{.3cm}}
\def\text{\mbox}
\def\rh2{{\mathbb R}{\mathbb H}^2}
\def\ch2{{\mathbb C}{\mathbb H}^2}

\def\CC{\operatorname{Core}}

\def\RP2{{\mathbb{RP}}^2}
\def\RP3{{\mathbb{RP}}^3}

\def\int{\operatorname{int}}
\def\PGL3{PGL(3{,\mathbb R})}
\def\PGL4{PGL(4{,\mathbb R})}

\def\H2R{{\mathbb H}^2\times {\mathbb R}}

\def\diam{\operatorname{diam}}
\def\stab{\operatorname{Stab}}
\def\vol{\operatorname{vol}}
\def\incl{\operatorname{incl}}

\def\Dcal{\mathcal D}

\def\cl{\rm{cl}\thinspace}

\def\dim{\rm{dim}}

\def\Scal{{\mathcal S}}
\def\Qcal{{\mathcal Q}}
\def\Xcal{{\mathcal X}}
\def\Fcal{{\mathcal F}}
\def\Ccal{{\mathcal C}}
\def\Acal{\mathcal A}
\def\Hcal{\mathcal H}
\def\Bcal{\mathcal B}
\def\Lcal{\mathcal L}

\def\Mcal{\mathcal M}

\def\Ncal{\mathcal N}
\def\Mcal{\mathcal M}

\def\Pcal{\mathcal P}
\def\Rcal{\mathcal R}
\def\Vcal{\mathcal V}
\def\Wcal{\mathcal W}
\def\Jcal{\mathcal J}
\def\z2{{\mathbb Z}/2}
\def\coker{\operatorname{coker}}
\def\cl{\operatorname{cl}}
\def\CH{\operatorname{CH}}
\def\prefab{Z}

\newtheorem{theorem}{Theorem}[section]
\newtheorem{lemma}[theorem]{Lemma}
\newtheorem{corollary}[theorem]{Corollary}
\newtheorem{proposition}[theorem]{Proposition}

\newtheorem{definition}[theorem]{Definition}
\newtheorem{claim}[theorem]{Claim}

\begin{document}\today

\title{Finite-volume Hyperbolic 3-Manifolds  contain immersed Quasi-Fuchsian surfaces. }
\author{Mark D. Baker}
\author{Daryl Cooper}

\address{IRMAR,
Universit\'e de Rennes 1,
35042 Rennes Cedex, FRANCE}
\address{Department of Mathematics, University of California, Santa Barbara, CA 93106, USA}

\email[]{mark.baker@univ-rennes1.fr}
\email[]{cooper@math.ucsb.edu}

\begin{abstract} The paper contains a new proof 
that a complete, non-compact hyperbolic $3$-manifold $M$ with finite volume contains an immersed, closed,
quasi-Fuchsian surface. \end{abstract}
\thanks{Authors partially supported by NSF grant 1207068.
Cooper also by NSF 0706887, 1065939, 1045292;  
and DMS 1107452, 1107263, 1107367 {\it RNMS: GEometric structures And Representation varieties} the {\it GEAR} Network,
and the CNRS
}

\thanks{ \today}

\maketitle
%%% Start main body of article
%
%   Intro
%
%

A {\em complete finite-volume hyperbolic $3$-manifold with cusps} is a non-compact hyperbolic $3$-manifold
with finite volume and universal cover hyperbolic space.
We give a new proof of the following result of Masters and Zhang \cite{MZ1}, \cite{MZ2}. 
% Their original proof is quite long and technical.

\begin{theorem}\label{quasifuchsiansurfacegroups} Suppose $M$ is a complete finite-volume hyperbolic $3$-manifold with cusps. Then  
there is a $\pi_1$-injective immersion $f:S\longrightarrow M$ of a closed, orientable surface $S$ with genus at least $2$
and $f_*(\pi_1S)$ is a quasi-Fuchsian subgroup of $\pi_1M$.\end{theorem}

 Cooper, Long and Reid \cite{CLR} showed
 that such manifolds contain geometrically finite closed surface groups but there might be accidental parabolics.
 Kahn and Markovic \cite{KahnM} have shown that a {\em closed} 
 hyperbolic 3-manifold contains an immersed QF ({\em quasi-Fuchsian}) surface. 
  
A  {\em prefabricated} 3-manifold, $\prefab$, is the union of
a finite number of convex {\em pieces}, each of which is either a rank-$2$ cusp or a QF 
   manifold $Q_i$ with rank-$1$ cusps.  We require {\em simple combinatorics}:   there are 
exactly two rank-1 cusps, with slopes that intersect once, inside each rank-2 cusp of $\prefab$.
See (\ref{prefab}) for the precise definition.
   The convex combination theorem  \cite{BC1} is used to ensure $\prefab$ has a convex thickening, $\CH(\prefab)$. 
   In this case $\partial \prefab$ consists of closed incompressible surfaces without parabolics. The main theorem
   follows from (\ref{fabricatedcover}) which says there is a covering space of $M$ 
   that has a convex core which is a prefabricated manifold.
   This construction of  QF surfaces is  similar to the method used in \cite{BC1}
   and  \cite{CL1}.

 The crucial step is to control how the QF manifold pieces of $\prefab$ intersect. In section (3) we study
 the intersection  $Q_1\cap Q_2$ of two QF manifolds with cusps.
  This is governed by a finite collection of convex subsurfaces
 immersed by local isometries into $\partial Q_i$. A compact core of $Q_i$ is homeomorphic to an interval times a compact
   surface $F$.
  A {\em spider} is a 
   compact subsurface $X\subset F$, satisfying certain conditions.
    After taking finite covers,  each component of $Q_1\cap Q_2$ is described by a spider.
    % and there are no triple intersections.

The crucial step relies on a result about surfaces:  the {\em spider theorem} (\ref{spidertheorem}). 
 Each component of $X\cap \partial F$ is an arc called a {\em foot} of the spider $X$.
%The conservative separability theorem \cite{BC2} is used to 
We show that 
if  every component of $\partial F$ contains at least one spider foot then,
after replacing $F$ by a suitable finite cover of $F$, and choosing certain lifts of the spiders, every boundary component of $F$ 
contains exactly one spider foot. 
This ensures the above mentioned {\em simple combinatorics} for $\prefab$. 
 
In section (\ref{compare}) we discuss the relation between our proof and that of Masters and Zhang.

\noindent{\bf Acknowledgements:}~ The authors thank {\em IHP} and Universit\'e de Rennes 1 for hospitality during completion of this work. The authors were partially supported by NSF grant 1207068.
Cooper also by NSF 0706887, 1065939, 1045292;  
and DMS 1107452, 1107263, 1107367 {\it RNMS: GEometric structures And Representation varieties} the {\it GEAR} Network, and the CNRS

%
%
%  Prefabricated manifolds
%
%
\section{Prefabricated $3$-manifolds}

In this section we define prefabricated $3$-manifolds and use the convex combination theorem to
 make them convex. We show the boundary consists of incompressible surfaces without parabolics.
 First we review some material about convex hyperbolic manifolds; see section 2 
of \cite{BC1} for further discussion.

\gap The following definition is {\bf not standard}.
A {\em hyperbolic manifold} is a  smooth $n$-manifold,
possibly with boundary, equipped with a metric so that every point has a neighborhood
that is isometric to a subset of hyperbolic space, ${\mathbb H}^n$.
An example is a compact annulus in ${\mathbb H}^2$.
A connected hyperbolic $n$-manifold $M$ is {\em convex} if every pair of points in the universal cover $\tilde{M}$ is connected
by a geodesic. It is {\em complete} if the universal cover is isometric to
${\mathbb H}^n$, and
 {\em metrically complete}
if every Cauchy sequence converges. 
 
 \gap

 If a hyperbolic $n$-manifold $M$ is convex, then
  the developing map embeds   $\tilde{M}$ isometrically into ${\mathbb H}^n$, and the covering transformations
of $\tilde{M}$ extend to give a group $\Gamma$ of  isometries of ${\mathbb H}^n$,
and $M$ is isometric to a submanifold of ${\mathbb H}^n/\Gamma$.  If $M$ is convex and 
$f:M\longrightarrow N$
is a local isometry into a hyperbolic $n$-manifold $N$, then $f$ is $\pi_1$-injective. 

  A hyperbolic $n$-manifold, 
$N$, is a {\em thickening} of a connected hyperbolic $n$-manifold, $M,$ 
  if $M\subset N$ and $\incl_*:\pi_1M\rightarrow\pi_1N$ is an isomorphism.  
 If, in addition, $N$ is convex then  $N$ is called  a {\em  convex thickening} of $M.$
 
 If $M$ is a subset of a metric space $N$, the {\em $\kappa$-neighborhood} of $M$ in 
$N$ is  $${\mathcal N}_K(M;N)=\{x\in N\ :\ d(x,M)\le \kappa\ \}$$
 If $M$ is a disjoint union of convex hyperbolic manifolds $M_i$, and ${\kappa}\ge0$,
the {\em $\kappa$-thickening} of $M$ is the disjoint union of the  convex thickenings of the components:
\[Th_{\kappa}(M)\ =\ \sqcup_i\ {\mathcal N}_{\kappa}(\tilde{M_i};{\mathbb H}^n)/\pi_1M_i\]

A {\em horocusp} is $C=\Bcal/\Gamma$ where
 $\Bcal\subset {\mathbb H}^3$ is a horoball and $\Gamma$  is a discrete, rank-2 free-abelian group of parabolics that
preserve $\Bcal$. Thus $\partial C=\partial \Bcal/\Gamma$ is  a {\em horotorus}. 

A {\em finite-area Fuchsian group} is a subgroup $\Gamma_F\subset Isom({\mathbb H}^2)$
such that $F={\mathbb H}^2/\Gamma_F$ is an orientable,  
hyperbolic surface with finite area. This is sometimes called
a {\em finitely generated Fuchsian group of the first kind}. 
Throughout this paper Fuchsian groups have finite area. 
An essential loop in a hyperbolic surface  is {\em peripheral} if it is  freely 
homotopic into the boundary, or into a cusp.
Since $F$ has finite area, every peripheral loop in $F$
has parabolic holonomy.

We fix an embedding ${\mathbb H}^2\subset{\mathbb H}^3$ and use this to
 identify $Isom({\mathbb H}^2)$ with a subgroup of $Isom({\mathbb H}^3)$.
Then there is a corresponding {\em Fuchsian $3$-manifold} $M_F={\mathbb H}^3/\Gamma_F$
which contains $F$ as a totally geodesic surface.

A {\em QF (quasi-Fuchsian) group}
is a subgroup $\Gamma\subset Isom({\mathbb H}^3)$ such that 
$M_{\Gamma}={\mathbb H}^3/\Gamma$ is a hyperbolic $3$-manifold
that is bilispchtiz homeomorphic to a Fuchsian $3$-manifold. 
A  $3$-manifold  is {\em QF} if it is convex and the holonomy is
a QF group. 

\begin{definition}\label{prefab}
A {\em prefabricated manifold} is a connected, metrically complete, finite-volume, hyperbolic $3$-manifold 
$$\prefab=\Ccal\cup \Qcal_1\cup \Qcal_2$$
Each component of $\Qcal_i$ and of $\Ccal$ is a convex hyperbolic $3$-manifold called a {\em piece}. 
Each component of $\Qcal_i$ is a 
QF $3$-manifold with at least one cusp. 
%that is homeomophic to
%$F\times[0,1]$, where $F$ is a closed, connected, orientable surface with finitely
%many points deleted and $\chi(F)<0$. 
Each component of $\Ccal$ is  a horocusp. These pieces satisfy the 
following conditions for $i\in\{1,2\}$, and for each component $C$ of $\Ccal$:
\begin{enumerate}
\item[(P1)] $\Qcal_i\cap\Ccal$ is the disjoint union of all the cusps in $\Qcal_i$
\item[(P2)]  $\Qcal_i\cap \partial C$  is an annulus with core curve $\alpha_i(C)$  
\item[(P3)] $\alpha_1(C)$ intersects  $\alpha_2(C)$ once transversally
\item[(P4)] Each component of $\Qcal_1\cap \Qcal_2$ intersects $\Ccal$
\end{enumerate}
\end{definition}

In general $\prefab$ is not isometric to a submanifold of ${\mathbb H}^3/\Gamma$ for any Kleinian group $\Gamma$. 
Under additional hypotheses $\prefab$ has a convex thickening, (\ref{prefabconvex}).
A {\em complete prefabricated manifold} is a complete hyperbolic $3$-manifold which is a convex thickening
of a prefabricated manifold. The following is an immediate consequence of (\ref{makeprefab})

\begin{theorem}\label{fabricatedcover} Suppose $M$ is a finite-volume hyperbolic $3$-manifold with cusps. Then $M$ has a covering
space which is a complete prefabricated manifold.
\end{theorem}

The main theorem (\ref{quasifuchsiansurfacegroups})  follows from this and the fact that  a complete prefabricated manifold contains
a  surface group without parabolics (\ref{prefabboundary}). This gives
a surface subgroup of $\pi_1M$ which is not a virtual fiber because
$M$ has cusps. Since it has no parabolics it is QF by (\ref{3kindssurfacegroup}).
This theorem  can also be used to give another proof of the fact (\cite{BC1}, 9.4) that for every essential
simple closed curve $C\subset T$, where $T\subset M$ is a horotorus, there is an {\em essential} immersed surface in $M$
bounded by two copies of a finite cover of $C$.

The geodesic compactification of ${\mathbb H}^n$ is the closed ball $\overline{\mathbb H}^n={\mathbb H}^n\sqcup \partial{\mathbb H}^n$
where $ \partial{\mathbb H}^n=S^{n-1}_{\infty}$.
The {\em limit set} of a subset $A\subset{\mathbb H}^n$ is $\Lambda(A)=\cl(A)\cap  \partial{\mathbb H}^n$ and
 the {\em convex core} $\CC(A)\subset{\mathbb H}^n$ of $A$ is the 
  convex hull of $\Lambda(A)$. 
Thus $\CC(A)$ is empty iff  $\Lambda(A)$ contains at most one point. Moreover if $A$ is convex then $\CC(A)\subset \overline{A}$.
 
 If $M$ has a convex thickening,
  then the {\em convex core of $M$} is $\CC(M)=\CC(\tilde M)/\pi_1M$, and
 the {\em convex hull}, $\CH(M)$, of $M$ is the smallest convex manifold containing $M$.
 A hyperbolic manifold $M$ is
 {\em geometrically finite} \cite{Bowditch} if for all (or some) $\delta>0$
  the $\delta$-thickening of $\CC(M)$ has finite volume.

Suppose that $N$ is a hyperbolic manifold and 
$M\subset N$ is a submanifold. Given $\kappa>0$ we say that $N$ 
 {\em contains a $\kappa$-neighborhood} of $M$ if for every $p\in M$ 
and every tangent vector $v\in T_pM$ with $||v|| \le \kappa$ then $exp_p(v)\in N.$ 
 The next result gives conditions which ensure that a $3$-manifold $M=M_1\cup M_2$,
which is the union of two convex hyperbolic submanifolds $M_1$ and $M_2$, has
 a convex  thickening:

\begin{theorem}[convex combination theorem]\label{convexcomb} Suppose:
\begin{enumerate}
\item[(C1)]  $Y=Y_1\cup Y_2$ is a connected hyperbolic $3$-manifold which is the union of two convex $3$-submanifolds $Y_1$ and $Y_2.$
\item[(C2)]   $M=M_1\cup M_2$ is a connected hyperbolic $3$-manifold which is the union of two convex $3$-submanifolds $M_1$ and $M_2.$
\item[(C3)]  $Y_i$ is a thickening of $M_i$.
\item[(C4)]  $Y$ contains an $8$-neighborhood of $M.$
\item[(C5)]  $Y_i$ contains an $8$-neighborhood of $M_i\setminus (M_1\cap M_2)$.  
\item[(C6)]  {\em No bumping}: Every component of $Y_1\cap Y_2$ contains a point of $M_1\cap M_2.$
\end{enumerate}
  Then  $M$ has a convex thickening and $\CH(M)\subset N_8(M)\subset Y$. 
   \end{theorem}
 \begin{proof} By theorem 2.9 in 
 \cite{BC1}  $M$ has a convex thickening. Hence there is an isometric embedding
 of the universal cover $\tilde{M}\subset{\mathbb H}^3$.
 Claim (2.2) in the proof of that theorem establishes that if a geodesic segment
 $\gamma$ has endpoints  in $\tilde{M}$ then $\gamma\subset N_6(\tilde M)$.
 By lemma (\ref{convexhull}) below, $\CH(\tilde{M})\subset N_2(N_6(\tilde{M}))$. 
It follows that $\CH(M)\subset N_8(M)$.
 \end{proof}

We use this to show that if a prefabricated manifold $\prefab$ is contained in a much larger one that
is made of thickenings of the pieces in the original, then $\prefab$ has a convex thickening.
The number of connected components of a space $X$ is denoted $|X|$. 

\begin{corollary}\label{prefabconvex} Suppose $\kappa\ge 8k$ where $k=(|\Ccal|+|\Qcal_1|+|\Qcal_2|-1)$  and suppose
\begin{enumerate}
\item[(Z1)] $\prefab^{\kappa}={\mathcal C}^{\kappa}\cup \Qcal_1^{\kappa}\cup\Qcal_2^{\kappa}$  is a prefabricated manifold 
\item[(Z2)]  $\prefab={\mathcal C}\cup \Qcal_1\cup \Qcal_2$ is a prefabricated manifold contained in $\prefab^{\kappa}$
\item[(Z3)] $\Qcal_i^{\kappa}$ is a thickening of $\Qcal_i$ 
\item[(Z4)] $\Ccal^{\kappa}=Th_{\kappa}(\Ccal)$
\item[(Z5)] $\Qcal_i^{\kappa}$ contains a $\kappa$-neighborhood of $\Qcal_i\setminus \Ccal$
\item[(Z6)] Every component of $\Qcal_1^{\kappa}\cap \Qcal_2^{\kappa}$ contains a point of $\Qcal_1\cap \Qcal_2$.
\end{enumerate} Then $\prefab$ has a convex thickening that is a submanifold of $\prefab^{\kappa}$.
\end{corollary}
\begin{proof}

 There is a hyperbolic $3$-manifold
 $\Pcal_1$  whose components are convex
$$\Qcal_1\cup\Ccal\subset \Pcal_1 \subset N_{8|\Ccal|}(\Qcal_1\cup\Ccal)\subset Z^{\kappa}$$
obtained 
 by gluing the components of $\Ccal$ (which are rank-$2$ cusps)
 onto the rank-$1$ cusps in $\Qcal_1$ one at a time, and taking the convex hull of the result each time. 
  This involves applying (\ref{convexcomb})   $|\Ccal|$ times. 
 Each time we attached a cusp requires we thicken by $8$,
 thus $\Pcal_1\subset N_{8|\Ccal|}(\Qcal_1\cup\Ccal)$. 
 It is routine to check the hypotheses of (\ref{convexcomb})
 are satisfied at each step.

(P1) and (P2) imply each cusp of $\Qcal_1$ is contained in a unique component of
$\Ccal$, and each component of $\Ccal$ contains a unique cusp of $\Qcal_1$. 
By (Z3) and (Z5) the components of $\Qcal_1\setminus\Ccal$ are far apart, so each component of $\Pcal_1$ is a thickening of a single
 QF manifold in $\Qcal_1$ with a rank-$2$
 cusp glued onto each rank-$1$ cusp. 
 % see Figure (\ref{tubedsurface}.

 Next  do the same for $\Qcal_2$
 with another copy of $\Ccal$ to produce  $\Pcal_2$ with
 $$\Qcal_2\cup\Ccal\subset \Pcal_2 \subset N_{8|\Ccal|}(\Qcal_2\cup\Ccal)\subset Z^{\kappa}$$
 
 The final step is to glue  the components of $\Pcal_1$ and $\Pcal_2$ together. Clearly
 $|\Pcal_i|=|\Qcal_i|$, so this involves
 applying (\ref{convexcomb}) $(|\Qcal_1|+|\Qcal_2|-1)$ times.
 Since $\prefab$ is connected we can enumerate the connected components of $\Pcal_1\sqcup \Pcal_2$
 in a sequence so that the union of the components in every initial segment of the enumeration
 is {\em connected}.
 
Inductively on $m$ we have a connected convex manifold $M_1\subset N_{8(|\Ccal|+m-1)}(\prefab)$
 that contains the first $m$ components in the enumeration
 and set $M_2$ equal to the $(m+1)$'th component.  We apply
 (\ref{convexcomb})   with $Y_1=N_8(M_1)\subset Z^{\kappa}$ 
 and $Y_2=N_8(M_2)\subset Z^{\kappa}$. These are convex thickenings by (\ref{8lemma})
 hence properties (C1)-(C5)  hold.
 The  {\em no bumping 
 property} (C6) in (\ref{convexcomb}) follows
 from (Z6) and (P4). 
Then $M_1\cup M_2$ has a convex thickening  
$\CH(M_1\cup M_2)\subset N_8(M_1\cup M_2)\subset N_{8(|\Ccal|+m)}(\prefab)$. \end{proof}

It is routine to show:
\begin{lemma}\label{8lemma} Suppose $M\subset N$ are convex hyperbolic $3$-manifolds and $N$ is a thickening of $M$
and $N_8(\CH(M))\subset N$. Then $N_8(\CH(M))\cong N_8(\CH(\tilde M))/\pi_1M$ is a convex thickening
of $M$.
\end{lemma}

Recall: a group is {\em freely indecomposable} or {\em f.i.} if it is not the free product of two non-trivial groups.

\begin{proposition}\label{prefabboundary} If $\prefab$ is a prefabricated manifold, then $\partial \prefab$ is non-empty
and each component is a closed incompressible surface of genus at least $2$. 
Moreover no essential loop in $\partial \prefab$ is homotopic into a cusp of $\prefab$.
\end{proposition}
\begin{proof} The boundary of $\prefab$ contains a non-empty subset of $\partial \Qcal_i$ so is not empty. 
 If $\partial \prefab$ is compressible then $\pi_1Z$ is the free product of two non-trivial groups. We now
 show it is not.
 
By Kurosh's theorem \cite{Ku}, the free product
of two f.i. groups, neither of which is cyclic, amalgamated along a non-trivial subgroup is f.i., as is an
HNN extension of a non-cyclic f.i. group along a non-trivial subgroup.

A  {\em tubed surface} \cite{BC1} is a 2-complex formed by gluing a torus onto each boundary component
of a compact surface, with nonempty incompressible boundary, 
so that each boundary component is glued onto an essential simple closed curve in a distinct torus.
The fundamental group of a tubed surface  is  f.i. (exercise for the reader) and not cyclic. 

The prefabricated manifold $Z = \cup_i Y_i$ where each $Y_i$ is homotopy equivalent to a tubed surface.
Each component, $X,$ of $\Qcal_1\cap\Qcal_2$ is convex thus $\pi_1$-injective. Each component  $R\subset Y_i\cap Y_j$
is  formed by adding rank-2 cusps to some such $X$, and is thus $\pi_1$-injective. Moreover $\pi_1(R)$  contains
a  ${\mathbb Z}^2$ subgroup, and is thus not trivial. The gluings result in HNN extensions and amalgamated free products.
 Hence $\pi_1Z$ is f.i.
 
% It follows that if $Z=A\cup B$
%is a $3$-manifold and $\pi_1A,\pi B$ are f.i.  and every component of $A\cap B$ has non-trivial $\pi_1$ and is $\pi_1$-injective
%in $Z$ then $\pi_1Z$ is f.i. This applies to prefabricated manifolds because
%the fundamental group of a tubed surface  is {\em f.i.} (exercise for the reader) and not cyclic,
%and the amalgamating subgroups always contain ${\mathbb Z}^2$.

Suppose there is
an essential annulus $A$ in $\prefab\setminus{\rm int}({\Ccal})$ with  boundary $\partial A=\alpha\sqcup\beta$ where
 $\alpha\subset \partial \prefab$ and  $\beta\subset \partial C$ for some horocusp $C\subset \Ccal$.
%We may isotope $A$ transverse to $\Scal_i$ for $i=1,2$.
 By (P2)  $\Qcal_i \cap \partial C$ is an annulus and by (P3) the core curves $\alpha_1(C)$ and $\alpha_2(C)$ of these annuli
 have intersection number one.
  It follows that
$\beta$ has  intersection number $n\ne 0$ with at least one of these core curves. 
However $[\alpha]=[\beta]\in H_1(\prefab)$ and $n$  depends only on the 
homology class. Since $\alpha$ is disjoint from these surfaces, $n=0$, which
contradicts the existence of $A$.
\end{proof}

It follows from work of Bonahon and Thurston that:
\begin{theorem}\label{3kindssurfacegroup} 
Suppose $M$ is a complete hyperbolic $3$-manifold with finite volume. If $S$ is a  closed, 
orientable surface with $\chi(S)<0$ which is $\pi_1$-injectively
immersed in $M$ then either $S$ is a virtual fiber, or else $S$ is geometrically finite, in which case either it
is QF or some element of $\pi_1S$ is (an accidental) parabolic.
\end{theorem}

 \section{Coverings of surfaces containing immersed subsurfaces}
 
A {\em spider pattern} (\ref{spiderdata}) consists
of a pair of surfaces (possibly not connected) each equipped with
various immersed surfaces that are identified in pairs, and is used later to model
 how QF 3-manifolds intersect.
The main result of this section is  (\ref{spiderpatterntheorem}) which asserts the existence
of a finite cover of a spider pattern with certain properties. 
This follows easily
from (\ref{spidertheorem}) whose proof  occupies the bulk of this section. 

A path in a surface $F$ with endpoints in $\partial F$ is {\em essential} 
if it is not homotopic rel endpoints into the boundary of $F$.
A loop in $F$ is {\em peripheral} if it is freely homotopic into $\partial F$. 
A function $f:X\longrightarrow Y$ between metric spaces is a {\em local isometry} if $X$ has an open
cover such that the restriction of $f$ to each set in the open cover is an isometry onto its image.

\begin{definition} An {\em immersed spider} is $(F,X,f)$ where $F$ and $X$ are compact, convex, hyperbolic surfaces, 
and $f:X\longrightarrow F$ is a local isometry and 
\begin{enumerate}
\item[(I1)] Each component of  $f^{-1}(\partial F)$ is an arc (called a {\em foot} of the spider). 
 \item[(I2)] $X$ has at least 2 feet.
\item[(I3)] If $\gamma$ is an essential  loop in $X$ then $f\circ\gamma$ is
 not peripheral in $F$.
\item[(I4)] If $\gamma$ is an arc  in $X$ with endpoints on distinct feet then $f\circ\gamma$ is essential in $F$.\end{enumerate} 
\end{definition}
 Clearly $f^{-1}(\partial F)\subset\partial X$. A spider is called {\em degenerate} if $X$ is a disc with exactly two feet. 
  If $f$ is injective we identify $X$ with $f(X)$ and
   regard the spider as the {\em subsurface} $X\subset F$ and refer to $(F,X)$, or sometimes $X$ as an {\em (embedded) spider}.
 
 A  spider $X$ can be decomposed as $X=B\cup \Lcal$ where $\Lcal$ is a regular neighborhood of the feet of $X$ and $B$ is the
closure of $X\setminus \Lcal$ and is called the {\em spider body}. Each component $L$ of $\Lcal$ is a rectangle
 called a {\em leg} of the spider and contains a spider foot in the boundary.
%convex quadrilateral with one side  a spider foot
%and the opposite side is $B\cap L$. The two remaining sides of 
%   $\partial L$ are arcs in $\partial X$. 

\begin{definition}\label{immersedspiderdef} An  {\em immersed spider surface} is $\Scal=(\Fcal,\Xcal,f:\Xcal\longrightarrow \Fcal)$ such that
\begin{enumerate}
\item[(S1)] Each component of $\Fcal$ and $\Xcal$ is a compact, convex, hyperbolic surface.
\item[(S2)]  If $X\subset \Xcal$ and $F\subset\Fcal$ are components with $f(X)\subset F$ then $(F,X,f|X)$ is an immersed spider. 
\item[(S3)] ({\em Ample spiders}) $f^{-1}(C)\ne\phi$ for each component $C\subset\partial \Fcal$. 
\end{enumerate}
\end{definition}

If $f$ is injective we regard $\Xcal$ as a subset of $\Fcal$ then $(\Fcal,\Xcal)$ is called an {\em (embedded) spider surface}.
We say $\Scal$ is {\em connected} if $\Fcal$ is connected, and in this case $\Fcal$ will often be denoted by $F$.
The condition (S3) says each boundary component of $\Fcal$ 
contains the foot of some spider. 

\begin{definition} An embedded spider surface $(\Fcal,\Xcal)$ is called  {\em simple} if
$C\cap\Xcal$ is connected for each boundary component $C\subset\partial\Fcal$. 
\end{definition}
This means each boundary component of $\Fcal$ contains exactly one spider foot. 

\begin{definition}\label{spidercoverdef} Suppose $\Scal = (\Fcal,\Xcal,f)$ and $\tilde{\Scal} = (\tilde{\Fcal},\tilde{\Xcal},\tilde f)$ are immersed spider
surfaces.
A  {\em spider cover} $(p,q):\tilde\Scal\longrightarrow\Scal$   of {\em spider degree $d$} consists of covering space maps 
$p: \tilde{\Fcal} \rightarrow \Fcal$ and $q:\tilde{\Xcal}\longrightarrow \Xcal$ 
such that  $\tilde\Xcal$ is the disjoint union of $d$ copies of $\Xcal$ and $q$ is the natural projection,
 and the following diagram commutes
 $\quad\begin{CD}
\tilde\Xcal  @>\tilde f>> \tilde\Fcal\\
@VVqV  @ VVpV\\
\Xcal   @>f>> \Fcal
\end{CD}$

 \end{definition}

 The pair $(p,q)$ is called a {\em spider covering map}. For each component $X\subset\Xcal$
 the components of 
$q^{-1}(X)$ are called the {\em lifts} of $X$. 
Observe that if $f$ and $\tilde f$ are both injective then, after identifying the spiders with subsurfaces
of $\tilde\Fcal$ and $\Fcal$, we have  $q=p|\tilde{\Xcal}$. Since $\Fcal$ need not be connected,
$p$ might not have a well defined degree.
 It is important to check the condition (S3) {\em ample spiders} is
satisfied when
constructing spider covers. 

\begin{theorem}[spider theorem]\label{spidertheorem}
Suppose $\Scal=(F,\Xcal,f)$ is a connected, immersed, spider surface. 
Then there is  a connected, simple, embedded spider surface $\tilde\Scal=(\tilde F,\tilde\Xcal)$  which spider 
covers $\Scal$ %with some spider degree $d>0$. Moreover 
and $\tilde{F}\setminus\tilde\Xcal$ is connected and $|\partial \tilde F|$ is even.
\end{theorem}

\begin{proof} By (\ref{manyspiders}) there is a spider cover which is an {\em embedded} spider surface.
By (\ref{simplespiders}) there is a further cover by a  {\em simple} spider surface with the required properties.\end{proof}

\begin{definition}\label{spiderdata} An {\em immersed spider pattern} is 
 $P = (\Scal_1,\Scal_2, \tau)$ where 
 $\Scal_i=(\Fcal_i,\Xcal_i,f_i)$ is an immersed spider surface and $\tau : \Xcal_1\longrightarrow  \Xcal_2$ is a map
 called  {\em the pairing} that induces a bijection between components.
\end{definition}

In later sections the pairing models how QF 3-manifolds are glued along submanifolds.
If $f_1$ and $f_2$ are both injective we omit them from the notation and refer to an 
{\em embedded spider pattern} or just {\em spider pattern}.  

\begin{definition}\label{spiderpattern}
A spider pattern $\tilde  P = (\tilde{\Scal}_1,\tilde{\Scal}_2, \tilde \tau)$ 
{\em covers} an immersed spider pattern
$P = (\Scal_1,\Scal_1 ,\tau)$ if there are spider covers
$(p_i,q_i):\tilde\Scal_i\longrightarrow \Scal_i$ which are {\em compatible with the pairings} in the sense that
$q_2\circ \tilde\tau=\tau\circ q_1$. 
\end{definition}

%The immersed spiders in $\Scal$ and $\Scal'$ are called the {\em components immersed spiders} of $P$.
Given an immersed spider surface $(\Fcal,\Xcal,f)$  each connected component  $F\subset \Fcal$ 
determines an immersed spider surface called a {\em component spider surface}
$\Scal_F=(F,\Xcal_F,f|\Xcal_F)$ where
$\Xcal_F = f^{-1}(F)$. A spider pattern is {\em simple} if every component spider surface is simple.

Given $\delta>0$ an immersed spider $(F,X^{\delta},f^{\delta})$ is a {\em $\delta$-thickening} of another immersed spider $(F,X,f)$
if $X\subset X^{\delta}$, and $f^{\delta}$ is an extension of $f$, and taking appropriate lifts
to universal covers $\tilde f^{\delta}(\tilde X^{\delta})$ contains a $\delta$-neighborhood of $\tilde f(\tilde X)$
in $\tilde F$. An immersed spider pattern $P^{\delta}$ is a {\em $\delta$-thickening} of another immersed spider pattern $P$
if all the component spider surfaces of $P^{\delta}$ are $\delta$-thickenings of those of $P$. It is routine to check
that $\delta$-thickenings always exist.
The main result of this section is:
\begin{theorem}[Spider pattern theorem]\label{spiderpatterntheorem} Given
 an immersed spider pattern $P$ there is $d>0$ such that for all $\delta>0$ there 
 is a simple embedded spider pattern $\tilde P^{\delta}$ that spider
 covers $P^{\delta}$ with spider degree $d$. 
\end{theorem}
\begin{proof} By (\ref{spidertheorem}) for each component $F$ of $\Fcal_i$  there 
is a simple spider surface $\tilde\Scal(F)=(\tilde{F},\tilde\Xcal)$ which spider covers the component
immersed spider surface $\Scal_F$ given by $F$
with some spider degree  $d(F)>0$. Moreover $\tilde F\setminus\tilde\Xcal$ is connected and 
$|\partial \tilde F|$ is even.
 Let $d$ be the lowest common multiple of all the $d(F)$ for $F$ a component of $\Fcal_1\sqcup\Fcal_2$. Define $(\tilde{\Fcal}_i,\tilde\Xcal_i)$ to be the disjoint union of $d/d(F)$ 
copies of $\Scal(F)$ as $F$ ranges
over components of $\Fcal_i$. %Similarly define the embedded spiders $\tilde{\Xcal}_i\subset\tilde{\Fcal}_i$.
This determines  a spider pattern $\tilde P$  except for the pairing $\tilde\tau$. There are obvious
covering space projections to $P$.
Since every spider in $\Xcal_i$ has the same number, $d$, of lifts to $\tilde\Xcal_i$ 
 there is a  pairing $\tilde{\tau}$ of the spiders in $\tilde{P}$ that covers the pairing $\tau$.  
 It only remains to arrange the condition on $\delta$.
 After replacing $P$ by $\tilde P$ it suffices to prove the theorem in the case $P$ is a simple embedded spider pattern.
 
Given a simple embedded spider pattern $P$ there is $\delta$-thickening $P^{\delta}$ consisting of immersed spiders.
We show there is a simple embedded spider cover $\tilde P^{\delta}$ of $P^{\delta}$ 
with spider degree $1$ which is a conservative cover of each component surface. 
The spiders in $P^{\delta}$ are immersed, and might intersect.  The argument in the
first paragraph of the proof of (\ref{manyspiders})  shows there is a  conservative cover
of each component surface, $F$ of $P^{\delta}$, and pairwise disjoint embeddings of these thickened spiders. Doing this for each
$F$ gives a spider cover $\tilde P^{\delta}$ of spider degree $1$.

\end{proof}

We turn now to the proof of (\ref{manyspiders}).
A finite sheeted covering space $\tilde{F}$ of  a compact surface $F$ is {\em conservative} 
if $|\partial\tilde F|=|\partial F|$.
A map $f:S\longrightarrow F$ is a {\em virtual embedding} if there is a finite cover $p:\tilde{F}\longrightarrow F$
and a lift $\tilde{f}:S\longrightarrow \tilde{F}$ which is an embedding. At various times
we wish to lift an immersed surface to a finite cover so it is embedded and does not separate.

\begin{proposition}\label{convexlift} Suppose $F$ and $Y$ are
 two compact,  convex hyperbolic  surfaces.
 Suppose $f:Y\longrightarrow \int(F)$ is a local isometry 
  and $f_*(\pi_1Y)$ contains no peripheral element of $\pi_1F$. Then there is a
conservative cover $\tilde{F}$ of $F$ such that $f$ lifts to an embedding $\tilde{f}:Y
\longrightarrow\tilde{F}$
 and $\tilde{F}\setminus \tilde{f}(Y)$
is connected.
\end{proposition}
\begin{proof}  Choose a basepoint $y\in Y$ and use $x=f(y)$ as the base point for $F$. 
Define $H= f_*(\pi_1(Y,y))\subset\pi_1(F,x)$ and
let $p_Y:\tilde{F}_Y\longrightarrow F$ be the cover corresponding
 to $H$. The map $f$ is $\pi_1$-injective so 
 it lifts to a homotopy equivalence $\tilde{f}_Y:Y\longrightarrow\tilde{F}_Y$.
Since $Y$ and $F$ are  convex the developing map embeds the universal covers 
$\tilde{Y}\subset\tilde{F}\subset{\mathbb H}^2$. But $Y$ and $\tilde{F}_Y$ 
are the quotient of their universal covers by $H$
 and it follows that 
 $\tilde{f}_Y$ is injective. 
 %If $\tilde{f}_Y$ is not injective there is a geodesic $\gamma$ in $Y$ such that $\tilde{f}_Y\gamma$
 %is a loop. This is a non-trivial element of $H$
   
 Let $B\subset\pi_1(F,x)\setminus H$ be the set represented by loops based at $x$ of length at most $2\diam(Y)$.
 Then $B$ is finite.
 By the conservative separability theorem \cite{BC2}, there is a conservative cover $p:\tilde{F}\longrightarrow F$ and basepoint
 $\tilde{x}\in\tilde{F}$ covering $x$ with the following properties
\begin{itemize} 
\item[(i)]  There a compact connected $\pi_1$-injective subsurface $S\subset\tilde{F}$ with $p_*(\pi_1(S,\tilde{x}))=H$.
\item[(ii)] $p_*(\pi_1(\tilde{F},\tilde{x}))$ contains no element of $B$.
\item[(iii)] $\tilde{F}\setminus S$ is connected %and $\incl_*:H_1(S)\longrightarrow H_1(\tilde{F})$ is injective.
\item[(iv)] The covering is conservative.
\end{itemize}
The existence of $S$ implies  $f$ lifts to $\tilde{f}:Y\longrightarrow \tilde{F}$ with $\tilde{f}(y)=\tilde{x}$ and 
we claim $\tilde{f}$ is injective.

 Suppose $\tilde{f}(a)=\tilde{f}(b)$. In $Y$ 
 there are paths 
 $\alpha$ starting at $y$ and ending at $a$, and  $\beta$  
 starting at $b$ and ending at $y$ both of length at most $\diam(Y)$. 
  This gives two paths $\tilde{\alpha}=\tilde{f}\circ\alpha$ and $\tilde{\beta}=\tilde{f}\circ\beta$ in $\tilde{F}$.
 Then $\tilde{\alpha}\cdot\tilde{\beta}$ is a loop in $\tilde{F}$ based at $\tilde{x}$ and going through $\tilde{f}(a)$. 
 It projects to a loop $\gamma$ in $F$ based at $x$
 of length at most $2\diam(Y)$, so $[\gamma]\in H$.
  Hence $\gamma$ lifts to a loop  $\tilde{\gamma}_Y$ in $\tilde{F}_Y$ based at $\tilde{f}_Y(y)$.
 Since $\tilde{f}_Y$ is injective and covers $\tilde f$ this implies $a=b$ so $\tilde{f}$ is injective as asserted.
 
 It follows that $\tilde{f}(Y)$ is a regular neighborhood of convex core of $S$, and the remaining claims follow from (iii) and (iv).\end{proof}

\begin{lemma}\label{bandsumspiders}  Suppose $\Scal_1=(F,X_1,f_1)$ and $\Scal_2=(F,X_2,f_2)$ are immersed spiders. 
Then there is an immersed spider
$(F,X,f)$ called a {\em band sum of $\Scal_1$ and $\Scal_2$} such that
 $X$ is the union of regular neighborhoods of $X_1$ and $X_2$ which intersect along an arc.
Moreover $f|X_i=f_i$ and each foot of $X$ contains exactly one foot of $X_1\sqcup X_2$.
\end{lemma}
\begin{proof} There is a  rectangle $D$ which maps to a convex neighborhood 
of a long immersed geodesic
arc $\lambda$ connecting $X_1$ and $X_2$. For a suitable choice of $\lambda$
 there is a convex thickening, $X$, of $X_1\cup D\cup X_2$.
Details are left to the reader.
\end{proof}

\begin{lemma}\label{ringedspiders} Suppose $(F,X',f')$ is an immersed spider. Then there is
an immersed spider $(F,X,f)$ such that $X'\subset X$ and $f|X=f'$ and 
\begin{itemize}
%\item[(E1)] $X'$ and $X$ have the same number of feet.
%\item[(E2)] Every foot of $X$ contains exactly one foot of $X'$. 
\item[(E1)] Every component of $\partial X$ contains at most one foot of $X'$.
\item[(E2)] Every component  $D$ of $cl( X\setminus X')$ is a disc and $\partial D\cap\partial X\ne\phi$.
\end{itemize}\end{lemma}
\begin{proof} Glue an annulus $A(L)$ onto each leg $L$ of $X'$, such that $R=X'\cap A(L)$ is a rectangle  in $L$ 
 that separates the foot of $L$
 from the body of $X'$ and the closure
 of $A(L)\setminus R$ is a disc $D(L)$. The resulting surface satisfies (E1) and (E2). This can be done so that the result
 has a convex thickening for which there is an isometric immersion of $X$ into $F$
 extending $f'$. The core curve of $A(L)$ maps to a long immersed geodesic
 loop in $F$ which is not peripheral, but wraps many times around the boundary component 
 containing the foot of $L$. Details are left to the reader.
\end{proof}

The following implies there is a single conservative cover of a compact hyperbolic surface
$F$ such that finitely many immersed spiders in $F$ simultaneously lift to embeddings
that are non-separating.

\begin{theorem}[embedded spiders]\label{manyspiders} Suppose $\Scal=(F,\Xcal,f_{\Xcal})$ is a connected, immersed, spider surface.
There is a connected, embedded spider surface $\tilde\Scal=(\tilde F,\tilde \Xcal)$ which is a spider cover
of spider degree $1$ of $\Scal$. Furthermore
 $|\partial\tilde F|=|\partial F|$ and $\tilde F\setminus \tilde\Xcal$ is connected and
  $\beta_1(\tilde F)>\beta_1(\tilde\Xcal\cup\partial\tilde F)$.
 \end{theorem}
 \begin{proof}  
By banding the spiders of $\Xcal$ together using (\ref{bandsumspiders})  we obtain an immersed spider 
$(F,X',f')$ containing
 $\Xcal$.
Let $(F,X,f)$ be the immersed spider surface with $X'\subset X$ given by (\ref{ringedspiders}). 
Let $F^+$ be $F$ union a compact convex collar on each
component of $\partial F$. Then $X$ is immersed in the interior of $F^+$ so by (\ref{convexlift}) there is a conservative 
cover $\tilde{F}^+$ of $F^+$ and an embedded lift of $X$ to
$\tilde{X}\subset \tilde F^+$ with $\tilde F^+\setminus \tilde{X}$ connected. Thus $\tilde X$ is an embedded spider
in the subsurface $\tilde F\subset\tilde{F}^+$.

For each foot $A\subset \partial\tilde X\cap\partial\tilde F$  
there is a rectangle $L=L(A)\subset\tilde{F}^+$ with one side $A$
and the opposite side of $L$ is an arc  in $\partial \tilde F^+$. Gluing these onto $X$
 gives an embedded
 spider $\tilde{X}^+\supset \tilde{X}$ in $\tilde{F}^+$. 
These rectangles are
the legs of $\tilde{X}^+$ and $\tilde X$ is the body of $\tilde{X}^+$.
 There is a bit of {\em fussing} to arrange that $\tilde{X}^+$ is convex, however the argument below does not require this.

We claim 
$\tilde F\setminus \tilde X$ is connected.  There is a homeomorphism of pairs $(\tilde{F},\tilde{X})\cong (\tilde{F}^+,\tilde{X}^+)$,
so it suffices to show $\tilde F^+\setminus \tilde X^+$ is connected.
 Let $L$ be a leg of $\tilde X^+$ and $B\subset\partial \tilde X$  the component  that
 intersects $L$.  By (\ref{ringedspiders})(E1)   $B$ is disjoint from all the other legs of $\tilde{X}^+$.  
 The  arc  $B\setminus L$ connects the two sides of $L$ thus adding $L$ onto $\tilde{X}$ does not disconnect the complement.
 This proves $\tilde F\setminus \tilde X$ is connected.

There is a lift of $X'\subset X$ to 
$\tilde{X}'\subset \tilde{X}$ and  $\tilde{F}\setminus\tilde{X}'$ is connected because,
by (\ref{ringedspiders})(E2), there is a path connecting every point in $\tilde{X}\setminus\tilde{X}'$ to 
a point   $p\in \partial\tilde X$. Since $\tilde X$ is a spider  we may choose $p$ in the interior of $\tilde{F}$. 
Thus $p$ is connected by an arc in $\tilde{F}\setminus\tilde{X}'$ to a point in the connected set $\tilde F\setminus\tilde X$.

There is a lift of the subsurface $\Xcal\subset X'$ to $\tilde\Xcal\subset\tilde X'\subset\tilde F$ and $\tilde F\setminus\tilde\Xcal$
is connected because $\tilde X'$ is obtained by band-summing the components of $\tilde\Xcal$ and then taking the convex hull.
Shrinking the convex hull and then  deleting  these bands clearly leaves the complement, 
$\tilde F\setminus\tilde \Xcal$,  connected.

The condition on $\beta_1(\tilde F)$ can be ensured by using a conservative cover of very large degree $d$. 
The relation between Euler characteristic and degree of a cover implies we may make 
 $\beta_1(\tilde{F})$ as large as we wish. However since the cover
is conservative and spider degree $1$ it follows that $\beta_1(\tilde\Xcal\cup\partial\tilde F)$ is independent of the cover.\end{proof}

It remains to prove (\ref{simplespiders}). If $F$ is a compact surface with boundary, the {\em capped surface}
 $C(F)=F\cup\Dcal$ is the closed surface obtained by gluing a disc
onto each circle component of $\partial F$, and $\Dcal$ is the union of the closed discs.
If $\Xcal$ is a disjoint union of spiders embedded in $F$
then each component of $\Xcal\cap\partial F$ is an arc and the 
{\em capped spiders} $C(\Xcal)=\Xcal\cup\Dcal$ is a compact subsurface of $C(F)$.

\begin{definition} The {\em spider graph} of a spider surface $\Scal=(\Fcal,\Xcal)$ is a bipartite graph $G=G(\Scal)$ with a {\em spider vertex} 
 $v(X)$ for
  each component $X\subset\Xcal$ 
 and a {\em boundary vertex}
 $v(C)$ for
 each  component $C\subset \partial \Fcal$. There is an edge $e(A)$ for each foot $A\subset \Xcal$. The edge
 $e(A)$ connects $v(X)$ to $v(C)$ where $X\subset \Xcal$  and $C\subset\partial F$ are the
 components containing $A$.
   \end{definition}
   
%A surface has ample spiders if and only if every vertex of the spider graph has positive edge degree: 
%at least one incident edge.

Embed $G(\Scal)$ in $C(\Xcal)=\Xcal\cup\Dcal$ as follows.
If $D$ is a disc component of $\Dcal$ with $\partial D=C$ then the  vertex $v(C)$ of $G(\Scal)$ is mapped to a point
 in $D$. 
If $X$ is a spider then $v(X)$ is mapped to a point in the spider body $B=B(X)$ of $X$.
The edge $e(A)$ in $G(\Scal)$ with endpoints $v(X)$ and $v(C)$ 
corresponds to the leg $L$ of $X$ with $L\cap C=A$.
This edge
 is mapped to an arc 
$\gamma=\beta\cdot\lambda\cdot\delta$ in $C(\Xcal)$ 
that is the union of an arc $\beta\subset B(X)$ starting at $v(X)$ and ending on $L\cap B(X)$,
 an arc $\lambda\subset L$ connecting $B(X)\cap L$ and $L\cap C$,
  and an arc  $\delta\subset D$ connecting $L\cap C$ to  $v(C)$. 
  It follows that $\Scal$ is simple iff each component of $G(\Scal)$ contains a single spider vertex.

\begin{lemma}\label{addendum} Suppose $(F,\Xcal)$ is an embedded spider surface and $F\setminus \Xcal$ is connected.
Then
the natural map below is injective
 $$\sigma:H_1(G(F,\Xcal);\z2)\longrightarrow H_1(C( F);\z2)\longrightarrow H_1(C(F);\z2)/incl_*(H_1( \Xcal;\z2))$$ 
 Moreover if $\beta_1(F)>\beta_1(\Xcal\cup \partial F)$ then $\sigma$ is not surjective. 
\end{lemma}
\begin{proof} Suppose   $\beta\in Z_1(G(F,\Xcal);\z2)$ with
 $0\ne [\beta]\in H_1(G(F,\Xcal);\z2)$. There is an edge $e$ of $G(F,\Xcal)$ with coefficient $1$ in $\beta$.
Let $A\subset\partial\Xcal$ be the foot corresponding to $e$.
Since
$F\setminus\Xcal$ is connected, there is an embedded loop  $\alpha \subset F$ such
that $\alpha\cap \Xcal = A$.  The algebraic intersection of $[\alpha]$ and $[\beta]$ is $1$ thus
$0\ne [\beta]\in H_1(C( F);\z2)$. Every element of $H_1( \Xcal;\z2)$ has intersection number  $0$
with $[\alpha]$. This is because $\Xcal$ can be isotoped into its interior and is then disjoint from $\alpha$.
 Thus $\sigma([\beta])\notin H_1( \Xcal;\z2)$, so $\sigma$ is injective.  
 
 $$\begin{array}{rcl}
 \dim[\coker\sigma] & \ge & \beta_1(C(F))-\left[\beta_1(\Xcal)+\beta_1(G(F,\Xcal))\right] \\
  & \ge & \beta_1(F) -\left[\beta_1(\partial F)+\beta_1(\Xcal)+\beta_1(G(F,\Xcal))\right]\\
  & = & \beta_1(F)-\beta_1(\Xcal\cup \partial F)
  \end{array}$$
 The additional hypothesis ensures this is positive.
 \end{proof}

A {\em morphism}
 between spider graphs is a simplicial map which preserves the type of each vertex, and
 is an {\em isomorphism} if it is also bijective.
Let $\ell(G)$ be the number of edges in a shortest circuit in $G$. Since $G$ is bipartite, $\ell(G)$ is even.
If there are no circuits $\ell(G)=\infty$.
The {\em boundary} of a spider surface $\Scal=(\Fcal,\Xcal)$ is the boundary
of the underlying surface: $\partial\Scal=\partial\Fcal$

\begin{lemma}\label{spidergraphcover} Suppose $\Scal=(F,\Xcal)$ and $\tilde\Scal=(\tilde F,\tilde\Xcal)$ are connected, embedded,
 spider surfaces. Then a spider cover $(p,q):\tilde{\Scal}\longrightarrow\Scal $  with spider degree $d$ induces a morphism
$p_G:G(\tilde\Scal)\longrightarrow G(\Scal)$ and:
\begin{itemize}
\item[(G1)] If $d=1$ and $p$ is a conservative cover then $p_G$ is an isomorphism.
\item[(G2)] If $d=1$ then  
$\ell(G(\tilde\Scal))\ge \ell(G(\Scal))$.
\item[(G3)] If   $p|\partial C$ is injective for each component $C\subset\partial \tilde F$
and  $\tilde\Xcal=p^{-1}\Xcal$,  then $p_G$ is a covering space projection.
\end{itemize} 
\end{lemma}

\begin{proof}  (G1) is obvious. (G3)
follows from the fact the spider graph $G(\Scal)$ embeds in $F$ and since both spiders and components of
$\partial F$ lift, it follows that we can choose embeddings with $G(\tilde{\Scal})=p^{-1}(G(\Scal))$. 

For (G2): if the cover is conservative the result follows from (G1).
Otherwise if the cover is not conservative, then $|\partial\tilde{F}|> |\partial F|$
and $G(\tilde\Scal)$ has more vertices corresponding to components of the boundary
than $G(\Scal)$. Clearly $p_G$ is a bijection
on the interiors of edges and on vertices corresponding to spiders. 
However if $C$ is a component of $\partial F$ then the pre-image of $v(C)$ has one vertex for each component of $p^{-1}(C)$.
One may regard $G(\tilde\Scal)$ as obtained from $G(\Scal)$ by cutting into several pieces some of the vertices of $G(\Scal)$
and attaching the edges to the resulting subdivided vertices in some way.
%, if $f:\tilde{F}\longrightarrow F$ is an $n$-fold
%cover such that every component of $\partial\tilde F$ projects homeomorphically to a component of $\partial F$and if $\tilde\Xcal=p^{-1}\Xcal$  then the covering map induces a $n$-fold covering of spider graphs $G(\tilde\Scal)\longrightarrow G(\Scal)$.
\end{proof}

\begin{theorem}[simple spiders]\label{simplespiders}
Every connected embedded spider surface $\Scal=(F,\Xcal)$  is spider covered by a
simple spider surface $\tilde\Scal=(\tilde F,\tilde\Xcal)$ such that $\tilde F\setminus\tilde\Xcal$ is
connected and $\tilde F$ has an even number of boundary components. \end{theorem}
 
 \begin{proof}  
  For each component $C$ of $\partial F$ the 
{\em excess number of spider feet on $C$} is $$e(C)=|\Xcal\cap C|-1$$ Condition (S3), ample spiders,
implies $e(C)\ge 0$ for all components $C\subset \partial F$.
The {\em excess number of spider feet on $F$} is $$e=e(F)=|X\cap \partial F|-|\partial F|=\sum_C e(C)$$
Observe $e(C)+1$ is degree of the vertex $v(C)$ and is therefore determined by the spider graph $G(\Scal)$. 
 The spider surface is simple iff $e=0$.

\begin{claim} Every connected spider surface
 is spider covered by a spider surface $\Scal=(F,\Xcal)$  with the following properties:
\begin{enumerate}
\item[(F1)] $F\setminus\Xcal$ is connected. 
\item[(F2)]   $|\partial F|\ge 4$ and is even.
\item[(F3)] $e=e(F)$ is even.
\item[(F4)]  $\ell(G(F,\Xcal))> 4$.
 \end{enumerate}
 \end{claim}

Properties (F2)-(F4) are determined by the isomorphism  type of the spider graph $G(F,\Xcal)$.
 Property (F4) says no spider has two feet on the same boundary component (no circuit of length 2) and two
spiders have feet on at most one common boundary component (no circuit of length 4).

\begin{proof}[Proof of claim]  The cover consists of a  sequence  of 5 spider covers (A),(B),(A),(B),(A) of two types called (A) and (B). 
A spider cover {\em preserves}  property (Fn) if whenever the original
spider surface has this property, so does the spider cover.

The type (A) cover is  a conservative cover of $F$ given by (\ref{manyspiders}) and thus has property (F1).
It is conservative and has spider-degree $1$, 
so by (\ref{spidergraphcover})(G1) it preserves the isomorphism type of the spider graph, and therefore
it preserves the remaining properties. 
 
The type (B) cover is the regular cover $p:\tilde{F}\longrightarrow F$ corresponding to the kernel of the natural surjection
 $$\pi_1F\longrightarrow H_1(C(F);\z2)/incl_*(H_1(\Xcal;\z2)).$$ 
 Every spider in $\Xcal$,
  and each boundary component of $F$, lifts for this cover,  and  $\tilde\Xcal = p^{-1}(\Xcal)$ is an {\em ample} collection of 
  disjoint spiders in $\tilde F$. 
 By (\ref{spidergraphcover})(G3)  the induced morphism $$p_G:\tilde{G}=G(\tilde F,\tilde\Xcal)\longrightarrow G=G(F,\Xcal)$$ is
a covering space projection.

Since each type (B) cover is always done just after a type (A) cover,
it follows from (\ref{manyspiders}) and  (\ref{addendum})  that $\sigma$ is injective but not surjective so $\tilde G$ consists of $2^k$ disjoint copies of 
the universal $\z2$-cover of $G$ with $k>0$. Hence the number of vertices (and hence $|\partial F|$)
 and $e(F)$ are all multiplied by  $2^m$ where $m=k+\beta_1(G(\Scal))>0$.
Thus spider covers of type (B) preserves properties (F2),(F3).

We assert that $\ell(\tilde G)=2\ell(G)$. Suppose $\alpha$ is an essential loop in $\tilde{G}$
of minimal length. Then it is a simple closed curve. It projects to an essential loop $\beta$ in $G$
which crosses each edge of $G$ an even number of time because it lifts to the loop $\alpha$.
 Hence the restriction of $p_G$ to $\alpha$ is a $2$-fold covering of $\beta$, which proves the assertion.
 
The initial graph is bipartite so initially $\ell\ge 2$. After doing a type (B) cover twice, $\ell\ge 8$ and $|\partial F|\ge 4$ and $e(F)$ is even.
The final type (A) cover restores (F1). This proves the claim.\end{proof}

We replace the original spider surface by one with the above properties and show that if $e> 0$ then there is a 
spider cover that reduces $e$ by 2  and continues to have these properties. Continuing reduces $e$
to $0$ which is  a spider surface that spider
covers the original and has  properties (F1) and (F2)  proving the theorem.

If $e\ne0$,  then $e\ge 2$ by (F3) and there are two cases to consider:\\
{\bf Case 1} there are  components  $C\ne C'$ of $\partial F$ with $e(C)\ge 1$ and $e(C')\ge1.$\\
{\bf Case 2} there is a  component $C$ of $\partial F$ with $e(C)\ge 2$. \\
Below we describe two spider covers of spider degree $1$
 that reduce $e$ by $2$ and increase $|\partial F|$ by 2, thus
they preserve (F2) and (F3). They preserve (F4) by  (\ref{spidergraphcover})(G2).
In both cases we follow the cover by a type (A) cover. The latter restores (F1), and gives an isomorphic spider graph, so it does not change $e$, and preserves (F2)-(F4).  

\medskip
{\bf Case 1}. Using (F1) and (F2)
  there is a $2$-fold cover $p:\tilde{F}\longrightarrow F$
such that every spider in $\Xcal$ lifts and $C$ and $C'$ are the only components of $\partial F$ with two disjoint
 lifts.

To construct this cover: by (F2) and (F1) we may choose a finite number of pairwise disjoint 
properly embedded arcs in $F$ which are disjoint from $\Xcal$
whose union has exactly one endpoint on each component of $\partial F$ except $C$ and $C'$. There is at least one such arc
 by (F2), so
these arcs represent a nontrivial element of $H^1(F;{\mathbb Z}/2)$ and determine $p$. In fact cross-joining
two copies of $F$ along this family of arcs gives the cover.

To construct a spider cover it remains
to choose one lift of each spider to obtain $\tilde{\Xcal}\subset \tilde{F}$. 
This must be done so $\tilde{\Xcal}$ has ample spiders (S3).
Then replacing $F$ with $\tilde{F}$ increases the number of boundary components of $F$ by
$2$ without changing the number of spider feet, so this reduces the excess by $2$. 

By (F4) no spider has two feet on $C$ so there are at least two distinct spiders $X_1$ and $X_2$  both with feet on $C$. 
Similarly there are $X_1'$ and $X_2'$ for $C'$. It is possible some $X_i$ equals some $X'_j$.
If this is the case we label so that $X_1=X'_1$. However $X_2\ne X_2'$ because $G(\Scal)$ contains no circuit of length $4$.

Since the covering is regular, at least one of the two  lifts of $X_i$ has a spider foot on a given
lift of $C$.
Choose lifts of $X_1$ and $X_2$  so that both of the boundary
components covering $C$ contain a spider foot. 
If $X_1\ne X_1'$ choose an arbitrary lift of $X_1'$. 
It is possible that the lift of $X_1'$ has spider feet on both lifts of $C'$. In this
case choose any lift of $X_2'$. Otherwise,
since $X_2\ne X_2'$,  we are free to choose  a lift of $X_2'$ which 
has a spider foot on the lift of $C'$
that does not  contain a spider foot on the chosen lift of $X_1$.
The remaining spiders
may be lifted in any way. This ensures the lifted spiders are ample (S3).

 \medskip
 {\bf Case 2.}
There is a $3$-fold cyclic cover $p:\tilde{F}\longrightarrow F$ such that every spider
in $\Xcal$ lifts, and the only component of $\partial F$ with more than one pre-image is $C$, and $C$ has
$3$ pre-images. 

To construct this cover,  since $|\partial F|\ge 4$ and is even, there is a finite
set of pairwise disjoint arcs properly embedded in $F$, so
 $C$ contains  one endpoint of each of exactly $3$ distinct arcs
  and every other component of $\partial F$ contains one arc endpoint. By  (F1) we may choose
  these arcs disjoint from $\Xcal$. 
Choose a transverse orientation on these arcs so that the arcs which meet $C$ induce the same orientation on $C$.
These transversally oriented arcs  represent an element of $H^1(F;{\mathbb Z}/3)$ and determine $p$. As before,
the cover can be constructed by cyclically cross-joining 3 copies of $F$ along these arcs.

By (F4) there are at least 3 distinct spiders with feet
on $C$. Choose lifts of these so that there is at least one spider foot
 on each pre-image of $C$. The remaining spiders
can be lifted in any way and the result is ample (S3). 
As before, replacing $F$ by $\tilde{F}$ reduces the excess by $2$.
 \end{proof}

\section{The intersection of Quasi-Fuchsian manifolds}

Suppose $Q_1$ and $Q_2$ are QF 3-manifolds embedded in a hyperbolic $3$-manifold $M$.
We assume the rank-1 cusps of $Q_1$ and $Q_2$ have different {\em slopes} in each rank-2 cusp of $M$.
Then each component $R$ of   
 $Q_1\cap Q_2$ is called an {\em ideal $3$-spider} (\ref{idealspider}) and is 
the union of a compact, convex manifold $R^c$ and finitely many ends called {\em legs} see (\ref{3spiderlemma}).
In  (\ref{gluingregion}) we generalize this when $Q_i$ are immersed in $M$ rather than embedded. 
This gives an immersed ideal $3$-spider
$R\looparrowright Q_i$. 

Next, (\ref{spiderapprox}) gives a two-dimensional approximation of this immersion 
by an immersion $X\looparrowright F_i$ with $F_i$ a finite
area hyperbolic surface with cusps, and $X$ is a convex  surface 
called an {\em ideal $2$-spider}.
An ideal $2$-spider is the union
of a compact convex part and finitely many ends called {\em legs}, each of which maps to a regular neighborhood
of a ray going out into a cusp of $F_i$. 

Truncating the cusps of $F_i$, and cutting the legs off the ideal spider $X$, 
and changing the metric  gives
 a (compact)
immersed spider as defined in section 2. 
Finally (\ref{simplewall})
shows how the problem of finding covers of QF $3$-manifolds
 with gluing regions that are far apart and with simple combinatorics
is related to the spider theorem.
In (\ref{Andsoma})  we relate spiders to some earlier work of Anderson and Soma.
  
  \gap

Suppose $\Bcal\subset{\mathbb H}^n$ is a horoball  centered on a point $x\in\partial{\mathbb H}^n$ bounded by the horosphere
 $\Hcal=\partial\Bcal$. A {\em vertical ray} is a ray in $\Bcal$ that starts on $\Hcal$ and limits on $x$. Given
$P\subset\Hcal$, {\em the set lying above $P$} is called a {\em vertical set} and
  is the union, $V(P)$, of the vertical rays starting on $P$. 
If $P$ is  convex, $V(P)$ is called a {\em thorn} and 
$P$ is called the {\em base of the thorn}. A thorn of dimension $2$ is also called a {\em spike}.
If $P=I\times{\mathbb R}$ is an infinite strip,  $V(P)$ is a {\em slab}

A hyperbolic $n$-manifold $E$ is an {\em excellent end} if it has finite volume and is isometric to $V/\Gamma$ for some
vertical set $V\subset\Bcal$ and discrete group $\Gamma\subset Isom({\mathbb H}^n)$ preserving $V$. 
The {\em horospherical boundary} of $E$ is $\partial_{\Hcal} E=(V\cap \Hcal)/\Gamma$.
An {\em excellent rank-1 cusp} is a $3$-manifold $V/\Gamma$ where $V$ is a slab and $\Gamma$ is a cyclic group of parabolics
preserving $V$.

A (possibly not connected) hyperbolic manifold $M$ is {\em excellent} if  $M=M^c\cup\Vcal_M$ where $M^c$ is
compact and $M^c\cap\Vcal_M= \partial_{\Hcal} \Vcal_M$ and each component of $\Vcal_M$ is an excellent end.
The pair $(M^c,\Vcal_M)$ is called an {\em excellent decomposition} of $M$. 
For example, an ideal convex polytope is excellent and the ends are thorns.
 Also, a complete hyperbolic $n$-manifold with finite volume is excellent since the ends are horocusps.
 Observe that an excellent manifold has finite volume.

If $N$ is a hyperbolic $3$-manifold and $S\subset N$ is an incompressible surface  
  with holonomy $\Gamma$,  then $S$ is a {\em QF surface}  
  if $M_S={\mathbb H}^3/\Gamma$ is
    QF. The convex core of $M_S$
    is a $3$-manifold unless $S$ is Fuchsian, in which case it is $S$.
    To overcome this mild
    technical irritation we define a convex $3$-manifold by $Q(S)=\CC(M_S)$ unless $S$ is
    Fuchsian, in which case $Q(S)=\CH(S\cup U)$ where $U\subset M_S$ is a small open set that meets $S$.  
     It is routine to show that if $S$ is a QF surface then $Q(S)$ has ends that
     are excellent rank-1 cusps thus $Q(S)$ is excellent.
     
      A compact, orientable surface properly embedded in a compact orientable $3$-manifold is {\em essential}  if  
it is  incompressible and $\partial$-incompressible. 

\begin{definition}\label{nicesurface} A  surface $S$  embedded
 in an excellent $3$-manifold  $M=M^c\cup\Vcal$ is {\em excellently essential} if each component of
 $S\cap\Vcal$ is an excellent annulus, and 
$S^c=S\cap M^c$ is a compact essential surface in $M^c$ with $\partial S^c\subset M^c$. 
\end{definition}

A {\em slope} on a torus is an isotopy class of essential simple closed curves. In view of the preceding, it makes
sense to talk about the {\em slope} of a excellently essential surface $S$ in a cusp of $M$, and the slope of a rank-$1$
cusp embedded in a rank-$2$ cusp.

\begin{definition}\label{idealspider}  An {\em ideal $n$-spider}
  is an excellent convex hyperbolic $n$-manifold $X$ with simply connected ends. Thus there is
  an excellent decomposition
  $X=B\cup {\Lcal}$ such that $B$ is compact and convex and each component of $\Lcal$ is
 a thorn. The components of $\Lcal$ are called  {\em legs} and  $B$ is called the {\em body}.
  \end{definition}

If the dimension $n$ is clear from context we will omit it and talk about an {\em ideal spider}. 
The definition implies that the holonomy of an ideal spider
 has no parabolics. A convex ideal polytope with $k$ ideal vertices is an ideal spider  with $k$ legs.
  An ideal spider is {\em degenerate} if it is a regular neighborhood of a geodesic.
 The following is obvious:

\begin{proposition}\label{3spiderlemma} Suppose $M$ is a complete hyperbolic $3$-manifold with finite volume
 and $Q_1,Q_2\subset M$ are excellent  QF
submanifolds. Then $Q_1\cap Q_2$ is excellent.
 If $Q_1$ and $Q_2$ have different slopes in every cusp of $M$, then each component of $Q_1\cap Q_2$
is an ideal spider.\end{proposition}

If $M$ and $N$ are excellent hyperbolic manifolds a map $f:M\longrightarrow N$ is {\em excellent}
if it is a local isometry and there are excellent decompositions with $f^{-1}(N^c)=M^c$. It follows
that each vertical ray in $\Vcal_M$ maps to a vertical ray in $\Vcal_N$. 

An {\em immersed QF manifold} is $(M,Q,f)$ where $f:Q\longrightarrow M$ is an excellent map between excellent hyperbolic
3-manifolds and $Q$ is QF. Two immersed QF manifolds $(M,Q_1,f_1)$ and $(M,Q_2,f_2)$ have {\em different slopes}
if for every cusp $V_i\subset Q_i$ whenever $f_1(V_1)$ and $f_2(V_2)$ are in the same cusp of $M$ then they have different slopes.
An  {\em immersed ideal $n$-spider} is $(M,R,p)$ where $M$ is an excellent $n$-manifold and $R$ is an ideal $n$-spider
 and $p:R\longrightarrow M$ is excellent.

Suppose $Q$ is an excellent QF 3-manifold
  and  $(Q,R,p)$ is an immersed ideal $3$-spider. We show in (\ref{spiderapprox}) that this is {\em approximated}
 by an immersed ideal $2$-spider $(F,X,f)$ for some complete hyperbolic surface $F$ with cusps.   

 If $N$ is a submanifold of a covering of a hyperbolic manifold $M$,
  the restriction of the covering space projection
 gives a local isometry $p:N\longrightarrow M$  called the {\em natural projection}.
  If $S$ is a QF surface in $M$ it is easy to see that the natural projection
 $Q(S)\longrightarrow M$ is excellent.
  The following generalizes (\ref{3spiderlemma}) to immersed QF manifolds.
    
     \begin{theorem}
     \label{gluingregion} 
   Suppose that $(M,Q_1,f_1)$  and { $(M,Q_2,f_2)$ are} two immersed QF manifolds with different slopes.
  Suppose $q_i\in Q_i$ and the {\em basepoint} $m=f_1(q_1)=f_2(q_2)$ is in a horocusp of $M$.
   
   Then there is a connected hyperbolic $3$-manifold $P=\tilde{Q}_1\cup\tilde{Q}_2$ where $p_i:\tilde{Q}_i\longrightarrow Q_i$ is a finite covering and $R=\tilde{Q}_1\cap\tilde{Q}_2$ is an ideal $3$-spider with at least 2 legs,
    thus $(Q_i,R,p_i|_R)$ is an immersed ideal spider. 
   
   The holonomy provides an identification of $\pi_1(M,m)$ with a Kleinian 
group $\Gamma\subset Isom({\mathbb H}^3)$. Then the QF manifolds
$Q_i$ have holonomy
 $\Gamma_i=(f_i)_*(\pi_1(Q_i,q_i))\subset\Gamma$ and the holonomy of $R$ is $\Gamma_1\cap\Gamma_2$.
      \end{theorem}
    \begin{proof} This is a special case of the virtual simple gluing theorem (4.3) in \cite{BC1}. For the convenience of the reader we include a self-contained proof.

 Let $Q'_i\subset{\mathbb H}^3$ be the embedding of the universal cover of $Q_i$ preserved by $\Gamma_i$ so $Q_i=Q'_i/\Gamma_i$. 
The set $R'=Q'_1\cap Q'_2$ is convex 
 hence so is the manifold $R=R'/\Gamma_R$ where $\Gamma_R=\stab(R')=\Gamma_1\cap\Gamma_2$.
% There is $\tilde{m}\in R'$ that projects to $m$. 
  We prove $\Gamma_R$
 is finitely generated. Hence  it is a separable subgroup of the free group $\Gamma_i$. 
 Since $R$ is convex it embeds in some finite covers of the $Q_i$.
 These coverings are then glued to produce $P$ by identifying the two copies of $R$.

 We first prove the corresponding statements for the compact cores obtained by removing the cusps
  and then deduce the result by gluing the cusps back on and 
  using the fact they are excellent.
  
 There are excellent decompositions with compact submanifolds $Q^c_i\subset Q_i$ and $M^c\subset M$
 with $m$ outside $M^c$ and
  $f_i^{-1}M^c=Q_i^c$. There is a natural projection $p_R:R\longrightarrow M$
   and we define $R^c=p_R^{-1}M^c$.
The pre-image, $Y\subset R'$, of $R^c$  is obtained from $R'$ by removing the intersections with the interiors
of the pairwise disjoint horoballs covering cusps in $M$. 
It follows that $Y$, and hence $R^c$, are connected.  If 
$U$ is a subset of a convex hyperbolic manifold $\prefab$ define  $N_1(U)=\Ncal_1(U,Th_1(\prefab))$.
We first show that  $R^c$ is 
compact by showing that $\vol(N_1(R^c))<\infty$. Observe that since $Q_i^c$ is
compact $\vol(N_r(Q_i^c))<\infty$ for all $r\ge0$.

There are natural projections 
$g_i:R\longrightarrow Q_i$ with $f_1\circ g_1=f_2\circ g_2$.
If $\vol(N_1(R^c))=\infty$ then, since $\vol(N_1(Q_1^c))<\infty$,
 there is a point $a\in N_1(Q_1^c)$ with infinitely many pre-images $\Acal=g_1^{-1}(a)$ in $N_1(R^c)$.
These project to the same point in $M$ hence there is $0<\delta<1$
 such that the $\delta$-balls in $N_2(R^c)$ centered on the points
of $\Acal$ are all pairwise disjoint.

The pre-image $\tilde\Acal\subset N_1(R')$ of $\Acal$ 
is contained in finitely many $\Gamma_i$-orbits; otherwise 
 $N_2(Q_i^c)$
contains infinitely many pairwise disjoint $\delta$-balls, contradicting it has finite volume.
Hence at least one of the orbits, $\Gamma_1\cdot\tilde{a}$, is infinite.
 But this orbit is contained in $N_1(Y)\subset N_1(({Q}_2^c)')$
and since $Q_2^c$ is compact, the set $\Gamma_1\cdot\tilde{a}$ 
is contained in finitely many $\Gamma_2$-orbits.
Hence there are two distinct points of $\Acal$ with pre-images in 
$\tilde\Acal$ that are in the same $\Gamma_R=\Gamma_1\cap\Gamma_2$ orbit.
But this means they have the same image in $\Acal\subset R=\tilde{R}/\Gamma_R$, 
a contradiction to the assumption that $|\Acal|=\infty$.
This proves the claim.  

Since $R^c$ is compact it follows that $\pi_1R=\Gamma_R$
 is finitely generated. The $3$-manifold $(Q_1'\cup Q_2')/\Gamma_R$  contains $R$ as a submanifold.
 Using subgroup separability in the free groups $\Gamma_i$  there are finite index subgroups $\Gamma_i'\subset\Gamma_i$
 giving finite covers 
$p_i:\tilde{Q}_i\longrightarrow Q_i$ and lifts $\tilde g_i:R\longrightarrow\tilde{Q}_i$ with $p_i\circ \tilde g_i= g_i$ and 
$\tilde g_i|R^c$ is injective.

A  hyperbolic 3-manifold $P$ is obtained from $Q_1'\cup Q_2'\subset{\mathbb H}^3$ by
using $\Gamma_i'$ to identity points in $Q_i'$. Let  $P^c$ be the submanifold that is the 
pre-image of $M^c$ under the natural
projection. Since $Q_i$ is excellent so is $\tilde{Q}_i$ and thus so is $P$. Since the ends of $P$ are vertical
it follows that $g_i$ is injective on all of $R$. Any identifications in the ends of $R$ would produce identifications
on $\partial R^c$ because the ends are excellent.

The hypothesis that the cusps of $Q_1=Q(S_1)$ and $Q_2=Q(S_2)$ always have different
slopes implies the ends of $R$ are thorns.  The spider, $R$,  has at least $2$ legs 
 because the basepoint $m$ 
  is in a cusp of $M$, so $R$ contains an essential arc in $S_1\cap S_2$ which contributes two legs.
\end{proof}
 
The manifold $R$ produced by this theorem is called a {\em gluing region} and the manifold $P$ is called the 
  {\em manifold obtained by  gluing $\tilde{Q}_1$ to $\tilde{Q}_2$ along $R$}. In general $P$ does not have a convex thickening.

The {\em Hausdorff distance} $\delta(A,B)=\delta_X(A,B)$ between two closed subsets $A,B\subset X$ of a 
metric space $X$ is the infimum of $K\in[0,\infty]$ such that $A$ is
 contained in a $K$-neighborhood of $B$ and $B$ is contained in a 
 $K$-neighborhood of $A$.  

The next result provides an immersed ideal $2$-spider $(F,X,g)$ that {\em approximates} an 
immersed ideal $3$-spider $(Q(S),R,p)$ in the sense that there is a bilipschitz homeomorphism between  universal covers
of $Q(F)$ and $Q(S)$ taking each pre-image of $X$ close (in the sense of Hausdorff distance) to a pre-image of $R$.

 \begin{proposition}\label{spiderapprox} Suppose $Q=Q(S)$ is QF 
 and  $(Q,R,p)$ is an immersed ideal $3$-spider with  $k\ge2$ legs. Then there is an
 immersed ideal $2$-spider $(F,X,f)$ with $k$ legs  that {\em approximates} it in the following sense.
  There  is a bilipschitz homeomorphism  $h:Q\longrightarrow Q(F)$ such that
 if  $\tilde{R},\tilde X,\tilde Q$ are universal covers and $\tilde{p}:\tilde{R}\longrightarrow\tilde{Q}$
 covers $p$ there is $\tilde f:\tilde{X}\longrightarrow\tilde{F}$ covering $f$
 such that   
$$\delta(\tilde{p}(\tilde R),\tilde h^{-1}\circ\tilde{f}(\tilde X))<\infty$$
Here we regard $F$ as a flat surface in $Q(F)$. Moreover $f_*(\pi_1X)=p_*(\pi_1R)$.
\end{proposition}
 \begin{proof} 
 By (\ref{QC})  there is a quasiconformal automorphism $H$ of $\overline{\mathbb H}^3$ that conjugates the
 holonomy, $\Gamma_S$, of $Q$ to the holonomy, $\Gamma_F$, of $F$ and
 a bilipschitz homeomorphism $h$ as required and $\tilde{h}=H|{\mathbb H}^3$.   
 We identify the universal cover $\tilde Q$ with  a subset of ${\mathbb H}^3$ and  $\tilde R$ with $\tilde p(\tilde R)\subset {\mathbb H}^3$.
  
 Define $Z=\CH(\Lambda(\tilde{R}))$ and $\tilde{X}=\CH(H(\Lambda(\tilde R)))$. 
Now $\Lambda(\tilde R)\subset\Lambda(\tilde Q)$ and
 $H(\Lambda(\tilde Q))=\Lambda(\tilde F)=\partial{\mathbb H}^2\subset\partial{\mathbb H}^3$.
 Thus
$H(\Lambda(\tilde R))\subset \partial{\mathbb H}^2$ hence
$\tilde{X}\subset {\mathbb H}^2$.   

Let $\Gamma_{\tilde{R}}\subset\Gamma$ be the stabilizer of $\tilde{R}$,
 so $\Gamma_{\tilde R}\cong\pi_1R$.   Then $\Gamma_{\tilde X}=H( \Gamma_{\tilde R}) H^{-1}$
 preserves $\tilde{X}$ and we obtain a hyperbolic surface $X=\tilde{X}/\Gamma_{\tilde{X}}$. 
Since $\Gamma_{\tilde X}\subset\Gamma_F$ there is
 natural projection $f:X\longrightarrow F$.
 
This identification
 of  $\tilde{X}$ with a subset of ${\mathbb H}^2\subset{\mathbb H}^3$
  makes $\tilde f$ the inclusion map, and we omit it in what follows, so that $\tilde X=\tilde f(\tilde X)$.
  With these identification we must show
 $\delta(\tilde R,\tilde h^{-1}(\tilde X))<\infty$. This follows from the next two claims.

\gap
{\bf Claim 1.} $\delta(\tilde{R},Z)<\infty$. 

Since $R$ is convex $Z\subset \tilde{R}$, so it suffices to show there is an upper bound on the distance of
points in $\tilde{R}$ from $Z$.
There is an upper bound on the distance of a point in a thorn from a geodesic ray running down the thorn.
 Since  $R=B\cup\Lcal$ is the union of a compact submanifold $B$ (the body of the spider)
  and $k\ge 2$ legs (which are thorns) there is $K>0$ such that every point  $x\in R$
 is within distance $K$ of some bi-infinite geodesic $\gamma=\gamma(x)$ in $R$ which starts in one leg
 of $R$ and ends in another. These geodesics lift into $Z$, proving the claim. 

\gap
{\bf Claim 2.}  $\delta(Z,\tilde{h}^{-1}(\tilde X))<\infty$.

By (\ref{convexhull}) every point in $Z$ is distance less than $2$ from a geodesic $\gamma$ with 
endpoints in $\Lambda(\tilde R)=\Lambda(Z)$.
Since $\tilde{h}$ is bi-Lipschitz, $\tilde{h}(\gamma)$
 is a quasi-geodesic, so
 there is $K>0$ independent of $\gamma$ such that $\tilde{h}(\gamma)$ lies within a distance
$K$ of a geodesic $\gamma'$ with endpoints in $\tilde{h}(\Lambda Z)$. 
Thus $\gamma'\subset \tilde{X}$
hence $\delta(\tilde{h}(Z),\tilde X)<\infty$.
Since $H$ is bilipschitz  $\delta(Z,\tilde{h}^{-1}(\tilde{X}))<\infty$. 

\gap
 If $R$ is a degenerate ideal $3$-spider it is easy to see that $\tilde{X}$ is a bi-infinite geodesic. In this case
 we thicken $\tilde{X}$  slightly in ${\mathbb H}^2$ to get a degenerate ideal $2$-spider.
 
\gap
{\bf Claim 3.} $(F,X,f)$ is an immersed ideal $2$-spider with $k$ legs. 

If $R$ is simply connected then, except in the degenerate case discussed above,
 $\tilde{X}\cong X$ is an ideal polygon, hence an ideal spider,
 with $k=|\Lambda(\tilde R)|$ vertices.
  In the general case we establish a similar picture in the covers
 $Q'$ of $Q$ and $Q'(F)$ of $Q(F)$  corresponding to $\pi_1R$.
 
Let $h':Q'\longrightarrow Q'(F)$ be the map covered by $\tilde h$. 
The projections of $\tilde R$ and $\tilde{X}$
give submanifolds $R'\subset Q'$  and $X'\subset Q'(F)$ homeomorphic to $R$ and $X$
 that are the images of {\em lifts}  of $p$ and $f$. 

Since $R$ and $Q'$ are convex and have the same fundamental group $Q'$ is a convex thickening of $R'$.
Similarly $Q'(F)$ is a convex thickening of $X'$.
 Since $R'$ is a lift of $R$ it follows that $R'=B'\cup\Lcal'$  is the union of
a compact body $B'$, and $k$ legs. There is a geodesic in $R'$
running from any leg to any other leg. Since $Q'$ is a convex thickening of $R'$ this
geodesic is distance minimizing between any pair of points on it.
Thus
the distance between distinct legs of $R'$ goes to infinity outside compact sets.
Since $h'$ is bilipschitz the image of a leg of $R'$ is contained in some $K$-neighborhood of a geodesic
ray in $Q'(F)$. Now $X'$ is a convex surface and $\delta(h'(R'),X')<\infty$, so $X'$ has
a leg (spike) close to the image of each leg of $R'$. Thus $X$ is the union of a compact subsurface
and $k$ spikes, hence a $2$-spider.  \end{proof}

Informally a {\em wall} is obtained from a QF manifold with finitely many immersed ideal $3$-spiders by deleting the cusps.
\begin{definition}\label{walldef} Suppose  
\begin{enumerate}
\item[(W1)] $Q=Q^c\cup\Vcal$ is an excellent QF manifold.
\item[(W2)] $\Rcal$ is the disjoint union  finitely many ideal $3$-spiders and $p:\Rcal\longrightarrow Q$ is excellent.
\item[(W3)] $(Q,R,p|R)$ is an immersed ideal $3$-spider for each component $R\subset\Rcal$.
\item[(W4)] ({\em Ample spiders}) $p^{-1}(V)\ne\phi$ for each component $V\subset \Vcal$. 
\end{enumerate}
Components of $\Rcal^c:=p^{-1}Q^c$ are called {\em gluing regions} and
 $W=(Q^c,\Rcal^c,p|:\Rcal^c\rightarrow Q^c)$ is called a {\em wall}.  
 The {\em base of the wall} is $\partial_{b}Q^c=Q^c\cap\Vcal$.
  A wall is {\em simple} if $p^{-1}(V)$ is connected for each component $V\subset\Vcal$.
  This means each component of the base of the wall intersects exactly one gluing region.
  \end{definition}

\begin{definition}\label{wallcoverdef} Suppose $W=(Q^c,\Rcal^c,p)$ and 
$\tilde W=(\tilde Q^c,\tilde\Rcal^c,\tilde p)$ are walls.
A  {\em wall cover} $(\pi,\pi'):\tilde W\longrightarrow W$   with {\em gluing degree $d$} consists of covering space maps 
$\pi: \tilde{Q}^c \rightarrow Q^c$ and $\pi':\tilde{\Rcal}^c\longrightarrow \Rcal^c$ 
such that  $\tilde\Rcal^c$ is the disjoint union of $d$ copies of $\Rcal^c$ and $\pi'$ is the natural projection,
 and following the diagram commutes
  $\quad\begin{CD}
\tilde\Rcal^c  @>\tilde p>> \tilde Q^c\\
@VV\pi' V  @ VV\pi V\\
\Rcal^c   @>p>> Q^c
\end{CD}$
 \end{definition}

We will show every wall is covered by a simple wall with embedded gluing regions.

\begin{definition}\label{approxdef} A connected spider surface $\Scal=(F,\Xcal,f)$ {\em approximates} a wall 
$W=(Q^c,\Rcal^c,p)$
if there is a diffeomorphism $h:I\times F\longrightarrow Q^c$ with $h(I\times\partial F)=\partial_{b}Q^c$
and the following holds. 
Let $\tilde{F}$ and $\tilde{Q}^c$ be the universal covers
and $\tilde h:I\times\tilde F\longrightarrow \tilde Q^c$ cover $h$.
Let $A$ denote the set of all submanifolds   $I\times\tilde f(\tilde X)\subset I\times\tilde F$ where $\tilde X$ is the universal cover 
of a component $X\subset\Xcal$ and $\tilde f$ covers $f|_X:X\longrightarrow F$.
 The action of $\pi_1F$ on $\tilde F$ induces an action on $A$.
 Let $B$ denote the set of submanifolds of $\tilde p(\tilde R^c)\subset\tilde Q^c$ where $\tilde R$
 is the universal cover of a component $R\subset\Rcal^c$ and $\tilde p$ covers $p:R\longrightarrow Q$.
   The action of $\pi_1Q^c$ on $\tilde Q^c$ induces an action on $B$.
   
We require there is $K>0$ and  a bijection $\theta:A\longrightarrow B$ such that
$\delta(x,\theta(x))<K$. This bijection is equivariant for the 
 actions  of $\pi_1F$ and $\pi_1Q$
via $h_*$. 
We also require that for every component $C\subset I\times\partial\tilde F$ and $x\in A$ that 
$x\cap C\ne\phi$ iff $\theta(x)\cap \tilde{h}(C)\ne\phi$.
  \end{definition} 
  
It is routine to check that if $\Scal$ approximates $W$ then a spider cover $\tilde\Scal$ induces a wall cover $\tilde W$
and $\tilde\Scal$ approximates $\tilde W$. Moreover $W$ is simple iff $\Scal$ is simple.

\begin{corollary}\label{spiderapproxwall}Every wall is approximated by an immersed spider surface.
\end{corollary}
  \begin{proof} By (\ref{spiderapprox}) an immersed ideal $3$-spider $(Q,R,p)$ is approximated
  by   an immersed ideal $2$-spider $(F,X,f)$.
  There are excellent decompositions $F=F^c\cup \Vcal_F$
  and $X=X^c\cup \Vcal_X$ where each component of $\Vcal_F$ is a cusp and 
  each component of  $\Vcal_X$
   is a spike and $f$ is excellent
  for this decomposition. There is a hyperbolic metric on $F^c$ so that $F^c$ has geodesic boundary
  and the pullback to $X^c$ using $f$ makes $X$ convex. With these new metrics
  $(F^c,X^c,f|)$ is an immersed  spider. It is routine to check the conclusion follows from (\ref{spiderapprox}).
 \end{proof}
 
An excellent convex  $n$-manifold $M=M^c\cup\Vcal$ is a submanifold of the complete manifold
$M^{\infty}=Th_{\infty}(M)$. Each component $V\subset\Vcal$ is covered by a vertical
subset of some horoball $\Bcal\subset {\mathbb H}^n$. The image of $\Bcal$ in $M^{\infty}$
is a vertical submanifold $V^{\infty}$ that is a thickening of $V$. Let $\Vcal^{\infty}\subset M^{\infty}$ be the union of all such,
then it is a thickening of $\Vcal$ and is the quotient by $\pi_1M$ of a collection of pairwise disjoint horoballs.
Given $\kappa\ge 0$   the {\em relative $\kappa$-thickening} $Th^{rel}_{\kappa}(M)$ of $M$ is the convex hull of
$X_{\kappa}(M)=Th_k(M)\setminus\Vcal^{\infty}$. It is the union of $X_{\kappa}(M)$ and excellent ends in $\Vcal^{\infty}$ 
that lie vertically above $Th_{\kappa}(M)\cap\partial\Vcal^{\infty}$.
It is a thickening of $M$ and it is excellent. It contains a
$\kappa$-neighborhood of $M^c$. If $W=(Q^c,\Rcal^c,p)$ is a wall and $\kappa>0$ then
the {\em $\kappa$-thickened wall} $W^{\kappa}$ consists of the relative $\kappa$  thickenings of $Q$ and $\Rcal$
truncated along the same cusps.
 
      The following allows us to reformulate the problem of finding a simple cover  of a wall  
 with lifts of gluing regions that are
far apart as a corollary of the spider theorem.
 
 \begin{corollary}\label{simplewall} Suppose that $W$ is a wall and $\Scal$ is a spider surface  that approximates $W$.
 Given $\kappa>0$ there is $\delta>0$ such that if $\tilde \Scal^{\delta}$ is a simple embedded
 spider cover of $\Scal^{\delta}$ then the corresponding cover $\tilde W^{\kappa}$ is simple and
 the gluing regions are pairwise disjoint.
 \end{corollary}
 \begin{proof} Refer to definition (\ref{spiderapprox}). If $\delta$ is sufficiently large then the corresponding wall
 cover $\tilde W^{\kappa}$ has embedded gluing regions. This is because if $R$ is an ideal spider
 then $Th_{\kappa}^{rel}(R)$ is convex so its
  universal cover  embeds isometrically  in ${\mathbb H}^3$. 
  The argument in the proof of (\ref{gluingregion}) implies if $\delta$ is large enough then the lifts
 of $X_{\kappa}(R)$ to $\tilde W$ corresponding to the lifted spiders in $\tilde\Scal^{\delta}$ are pairwise disjoint.
  The map $h$ is covered by a $K$-bilipschitz homeomorphism
 between $\tilde \Scal^{\delta}$ and $\tilde W^{\kappa}$. Hence for $\delta$ large
 the lifted gluing regions in $\tilde W$ are far
 apart. 
 \end{proof}

 \begin{lemma}\label{convexhull} If $R\subset\overline{\mathbb H}^3$ is any subset, then every point in  $\CH(R)$
 is within a distance $2$ of a geodesic segment with endpoints in $R$. 
 \end{lemma}
 \begin{proof} This easily follows from the fact that the convex hull of $R$
 is the union of the (ideal) simplices with vertices in $R$ and the thin triangles constant for hyperbolic space
 is less than $1$.\end{proof}
 
  \begin{theorem}\label{QC} Suppose $S$ is a QF surface, so $Q(S)$ is a finite volume
  convex submanifold of the complete QF manifold $M_S$.
Suppose  $F$ is a finite-area hyperbolic surface that is homeomorphic to $S$. 

Then there is a bilipschitz homeomorphism
 $h:M_{S}\longrightarrow M_F$   with $h(Q(S))=Q(F)$ and  a quasi-conformal 
 automorphism 
$H$ of $\overline{\mathbb H}^3$ such that $H|{\mathbb H}^3$ covers $h$.\end{theorem}
\begin{proof} This is well known except for the fact we may arrange $h(Q(S))=Q(F)$.
This follows from the fact there is a bilipschitz self-homeomorphism of $M_F$ 
that takes $Q(F)$ to $h(Q(S))$. This uses 
 that $Q(F)$ and $Q(S)$ are the union of a diffeomorphic compact
part and excellent cusps.
\end{proof}

\gap   
  {\bf Spiders in (relation to) the work of Anderson and  Soma} 
  
  We do not make use of the following.
Suppose $\Gamma_1,\Gamma_2$ are QF subgroups of a Kleinian group $\Gamma$ and
 that $x\in \partial{\mathbb H}^3$ is fixed by non-trivial parabolics $\gamma_1\in\Gamma_1$ 
 and $\gamma_2\in \Gamma_2$.
The subgroup $\Gamma'$ of $\Gamma$ generated by $\gamma_1$ and $\gamma_2$ is discrete and is 
free-abelian of rank $1$ or $2$. If $\Gamma'$ has rank-2 then $\gamma_1$ and $\gamma_2$ translate in different directions
in a horosphere ${\mathcal H}\subset{\mathbb H}^3$ centered at $x$ and represent
 different slopes on the quotient horotorus ${\mathcal H}/\Gamma'$. 
 Let $P(\Gamma_1,\Gamma_2)\subset  \partial{\mathbb H}^3$
denote the  (possibly empty) set of all such points. Anderson calls this the {\em exceptional set}.

The following is an immediate consequence of theorem C in Anderson \cite{Anderson}, see also Soma \cite{Soma}.
\begin{theorem}\label{Andsoma} Suppose $\Gamma_1,\Gamma_2$ are QF subgroups of
a Kleinian group $\Gamma$. Then $\Lambda(\Gamma_1)\cap\Lambda(\Gamma_2)=\Lambda(\Gamma_1\cap\Gamma_2)
\cup P(\Gamma_1,\Gamma_2)$.\end{theorem}

It follows that the universal cover of the convex core of
a gluing region is the convex hull of $\Lambda(\Gamma_1\cap\Gamma_2)\cup P(\Gamma_1,\Gamma_2)$.
The $(\Gamma_1\cap\Gamma_2)$-orbits of points in $P(\Gamma_1,\Gamma_2)$ correspond to the thorns 
forming the spider's legs.

\section{Constructing Prefabricated Manifolds}

In  (\ref{2slopes}) we construct the pieces that are used to build the prefabricated manifold $\prefab$.
These pieces are submanifolds of covering spaces of the manifold $M$ in theorem 
\ref{quasifuchsiansurfacegroups}. 
The main theorem  follows from (\ref{makeprefab}).

Two transverse excellently essential surfaces $J_1,J_2\subset M$ in a hyperbolic $3$-manifold $M=M^c\cup\Vcal$ 
have {\em essential intersections} if every component of $J_1\cap J_2\cap M^c$ is either a circle
that is not homotopic into $\partial M^c$ or an arc that is not homotopic rel endpoints into $\partial M^c$.

\begin{proposition}\label{2slopes} Suppose $M=M^c\cup\Vcal$ is an excellent decomposition of a complete, 
finite volume,  hyperbolic $3$-manifold with cusps
and $\partial M^c=T_1\sqcup\cdots\sqcup T_p$. 
Then there are transverse {\em excellently essential} surfaces $\Jcal_1,\Jcal_2\subset M$
such that every component of these surfaces is QF.
Moreover for every torus $T\subset\partial M^c$ both $T\cap\Jcal_1$ and $T\cap\Jcal_2$   are nonempty and have different slopes.
\end{proposition}
\begin{proof} Each homomorphism  $\rho:\pi_1M\longrightarrow SL(2,{\mathbb C})$  
determines a character  $\chi:\pi_1M\longrightarrow{\mathbb C}$ by
$\chi(\alpha)={\rm trace}(\rho\alpha)$.
The character variety $X$ is the set of all such characters. It is an affine algebraic variety over ${\mathbb C}$.
Let $X_0$ be the component of $X$ containing the character of the holonomy   
of the hyperbolic structure on $M$.
Thurston \cite{WPT1} proved that $X_0$ has  complex dimension $p$.

Choose a slope $\alpha_i$ on each $T_i$.  Let $Y=Y(\alpha_1,\cdots,\alpha_p)$ be the subset of $X_0$ 
defined by the $(p-1)$ polynomial equations $\chi^2(\alpha_1)=\chi^2(\alpha_{i})$
for $2\le i\le p$. Then $Y$ is a affine algebraic variety which contains
the character $\chi_0$ of $\rho_0$. This is because at the hyperbolic structure every slope is parabolic
so   $\chi_0^2(\alpha)=4$ for every slope $\alpha$ on every torus in $\partial M$. 

Thus $Y$ has  complex dimension  at least $p-(p-1)=1$. The function $f=\chi^2(\alpha_1)$
 is not zero at points on $Y$ close to $\chi_0$.
This is  because  a representation $\rho$ close to $\rho_0$ with $f=4$ is parabolic on each 
boundary component.  
 Therefore $\rho$ is the holonomy of a complete finite volume hyperbolic structure on $N$. 
 By Mostow-Prasad rigidity
  $\rho$ is conjugate to $\rho_0$. Since $Y$ has dimension at least $1$
   it follows that $Y$ has dimension $1$ and
   $f\ne 4$ in a small deleted neighborhood of $\chi_0$ on $Y$ as asserted.

Thus there is a discrete rank-1 valuation $\nu$ on $Y$ such that $\nu(f)<0$.
The Culler-Shalen machinery (\cite{CS1},\cite{CS2}, cf (9.2) of \cite{BC1}) applied to
$(Y,\nu)$ gives an action of $\pi_1M$ on a simplicial tree  and an 
essential surface $\Jcal_1=\Jcal(\alpha_1,\cdots,\alpha_k)$ dual to this action.
Since $\nu(f)<0$ each $\alpha_i$ acts on the tree without a fixed point. Therefore
this surface has non-empty intersection with every $T_i$ and the slope of $\Jcal_1$ on $T_i$ is some $\beta_i\ne\alpha_i$.
By surgering annuli, as in Lemma (2.3) of \cite{CL1}, 
we may arrange that every component of $\Fcal_1$  is QF.

Now repeat using $Y=Y(\beta_1,\cdots,\beta_p)$. This produces another essential QF surface
 $\Jcal_2=\Jcal(\beta_1,\cdots,\beta_p)$ with slope
$\gamma_i\ne\beta_i$ on $T_i$. It is routine to show  these surfaces can be isotoped to be transverse, and
excellently essential. \end{proof}

 \begin{theorem}\label{makeprefab} If $M$ is a complete, 
 finite volume,  hyperbolic $3$-manifold with cusps
  then there is a prefabricated
 $3$-manifold $\prefab={\Ccal}\cup\Qcal_1\cup\Qcal_2$ with a convex 
 thickening $\CH(\prefab)$ and a local isometry $g:\CH(\prefab)\longrightarrow M$.
  \end{theorem}
 \begin{proof} Choose an excellent decomposition $M=M^c\cup\Vcal$. Let $\Jcal_1,\Jcal_2\subset M$
be the excellently essential surfaces given by (\ref{2slopes}). In what follows $i\in\{1,2\}$. Then $\Jcal_i^c=\Jcal_i\cap M^c$
is a compact essential surface in $M^c$ and $\partial\Jcal_i^c$
 is a non-empty set of disjoint, essential, simple closed curves.
The set of intersection points $\Mcal =\partial \Jcal_1^c\cap \partial\Jcal_2^c$ between these curves is  finite.
 Since $\Jcal_1$ and $\Jcal_2$
 each meet every cusp of $M$, and
 have different slopes, $\Mcal$ contains at least one point
 on each component of $\partial\Jcal_i^c$. 
 
{ 
Define an equivalence relation
 on $\{1,2\}\times\Mcal$ by $[i,m]=[i',m']$ iff $i=i'$ and both $m$ and $m'$ are in the same component of $\Jcal_i$.
 Denote this component  $J_{[i,m]}$.
     Then $Q_{[i,m]}=Q(J_{[i,m]})$ is  an excellent convex QF manifold
    and there is a natural projection $f_{[i,m]}:Q_{[i,m]}\longrightarrow M$.
    After thickening if necessary, we may assume there is $q_{i,m}\in Q_{[i,m]}$ with $f_{[i,m]}(q_{i,m})=m$.
Applying  (\ref{gluingregion}) to $(M,Q_{[1,m]},f_{[1,m]})$ and $(M,Q_{[2,m]},f_{[2,m]})$ with 
base points $q_{1,m},q_{2,m}$ and $m$ gives an ideal $3$-spider $R_m$ and
 two immersed ideal $3$-spiders $(Q_{[i,m]},R_m,p_{i,m})$.  
   
Define $\Rcal=\{R_m:m\in\Mcal\}$
and $\Rcal_{[i,m]}\subset\Rcal$ be those ideal 3-spiders that are immersed in $Q_{[i,m]}$.
Observe that $\{\Rcal_{[i,m]}: m\in\Mcal\}$ is a partition of $\Rcal$ for each of $i=1$ and $i=2$.
 There is a wall $W_{[i,m]}=(Q_{[i,m]},\Rcal_{[i,m]},p^*_{[i,m]})$ where $p^*_{[i,m]}|R_k=p_{i,k}$  for
 each $R_k\in\Rcal_{[i,m]}$.    
Condition (W4) (ample spiders) is satisfied because $\Mcal$ contains at least one point
 on each component of $\partial\Jcal_i^c$. 
Define $\Wcal_i=\{ W_{[i,m]}:m\in\Mcal\}$ then 
 $|\Wcal_i|= |\Jcal_i|$. There is a natural homeomorphism $\sigma$
between the ideal spiders  in $\Wcal_1$ and those in $\Wcal_2$ that
sends the copy of $R_m$ in $W_{[1,m]}$ to the copy in $W_{[2,m]}$. This gives a {\em wall pattern} $(\Wcal_1,\Wcal_2,\sigma)$.

By (\ref{spiderapproxwall}) each wall $W_{[i,m]}$ is approximated by a connected immersed spider surface 
$(F_{[i,m]},\Xcal_{[i,m]},f_{[i,m]})$. Combining these we get two spider surfaces $\Scal_1,\Scal_2$ 
approximating $\Wcal_1$ and $\Wcal_2$.
 Moreover $\sigma$ determines a pairing so we
obtain a spider pattern $P=(\Scal_1,\Scal_2,\tau)$.

 Let $d=d(P)>0$ be the constant given by (\ref{spiderpatterntheorem}) for the spider pattern $P$.
  Define $\kappa=24\cdot d\cdot f$ where $f$ is the total number of feet of
  all the spiders $\sqcup_m \Xcal_{[1,m]}$. 
   For this value of $\kappa$ there is $\delta_m>0$ satisfying (\ref{simplewall}) 
     for $Q=Q_{[i,m]}$ and all the ideal spiders $R_m$ it contains, and $F=F_{[i,m]}$ 
    and all the immersed spiders $\Xcal_{[i,m]}$ it contains. 
    The relative thickenings in (\ref{simplewall}) of  $R_m$ are relative to the cusps of 
    $Q_{[i,m]}$ that are the pre images of $\Vcal$.
  Now set $\delta=\max\delta_m$.
 }
 
 By (\ref{spiderpatterntheorem}) there is a simple embedded spider { pattern}
 $\tilde{P}^{\delta}=(\tilde \Scal_1^{\delta},\tilde \Scal_2^{\delta},\tilde{\tau}^{\delta})$ that is a spider cover of spider degree $d$ of a $\delta$-thickening of $P$. 
 Here $\tilde\Scal_i^{\delta}=(\tilde{\Fcal_i},\tilde{\Xcal}_i^{\delta},\tilde{f}^{\delta}_i)$ 
 are the embedded simple spider surfaces of $\tilde{P}^{\delta}$.  
The number of spider feet   of $\tilde\Xcal_i^{\delta}$ is $d\cdot f$. Since the cover
is simple 
  $|\partial\tilde\Fcal_1|=d\cdot f=|\partial\tilde\Fcal_2|$. 

There are relative $\kappa$-thickenings $\Wcal_i^{\kappa}$ of $\Wcal_i$
and covers, $\tilde\Wcal_i ^{\kappa}$, corresponding to $\tilde P^{\delta}$.
By the choice of $\delta$ the component walls of $\tilde\Wcal_i ^{\kappa}$ are simple and the gluing regions are embedded.
The pairing determines a bijection between the lifted gluing regions of $\tilde\Wcal_1 ^{\kappa}$ 
and $\tilde\Wcal_2 ^{\kappa}$. These gluing regions are copies of elements of relative
thickenings of components of $\Rcal$. Corresponding gluing regions are isometric. This
gives a new wall pattern $(\tilde\Wcal_1^{\kappa},\tilde\Wcal_2^{\kappa},\tilde\sigma)$.
The component walls of this pattern are simple and the gluing regions are disjoint.

We now add back the cusps to the walls. 
Define $\Qcal_i^{\kappa}$ to be the disjoint union  $\sqcup_W\CH(W)$ over the walls $W$ in $\tilde\Wcal_i^{\kappa}$.
If we regard the gluing regions as submanifolds of the walls, their convex hulls are ideal $3$-spiders
in $\Qcal_i^{\kappa}$, and $\tilde\sigma$ gives a map between these submanifolds of 
$\Qcal_1^{\kappa}$ and $\Qcal_2^{\kappa}$
that is an isometry on each component. Identifying these submanifolds gives $Y=\Qcal_1^{\kappa}\cup\Qcal_2^{\kappa}$
where $\Qcal_1^{\kappa}\cap\Qcal_2^{\kappa}$ is the union of the ideal $3$-spiders.

Then we glue  on covers of components of $\Vcal$ to each end of $Y$ to obtain a prefabricated manifold
$Z^{\kappa}$.  Each rank-1 cusp of $\Qcal_1$
has been glued to exactly one  rank-1 cusp of $\Qcal_2$ along a thorn.  
These identifications are compatible with
the natural projections { $f_{[i,m]}$} so there is a local isometry $g:Y\longrightarrow M$.

Each end $E$ of $Y$ is a vertical set: it is the union, $B_1\cup B_2$, of 
two vertical rank-1 cusps $B_i\subset\Qcal_i$ and 
$B_1\cap B_2$ is a thorn. 
Thus $E$ is diffeomorphic to the product of a ray and a torus minus an open parallelogram.
The end $E$ projects into a rank-2 cusp  $C\subset \Vcal$. There is a unique
finite cover $\tilde{C}$ of $C$ so that this projection lifts to an isometric embedding. We use this embedding
to glue $\tilde{C}$
 onto $E$ and do this for each end $E$ to obtain $\prefab^{\kappa}$.
Define $\Ccal^{\kappa}$ to be the disjoint union of these $\tilde{C}$. 
The fact the walls are simple ensures (P1)-(P4) thus $\prefab^{\kappa}$ is a prefabricated manifold
$Z^{\kappa}=\Qcal_1^{\kappa}\cup\Qcal_2^{\kappa}\cup\Ccal^{\kappa}$.

Each component of $\Qcal_i^{\kappa}$ contains at least one gluing region. Each gluing region
corresponds to at least $2$ spider feet, so  $|\Qcal_i|^{\kappa}\le d\cdot f$. Also since each component of 
$\Ccal^{\kappa}$ corresponds to a spider foot so
$|\Ccal^{\kappa}| \le d\cdot f$. Hence $k:=|\Qcal_1^{\kappa}|+|\Qcal_2^{\kappa}|+|\Ccal|-1\le 3\cdot d\cdot f$. Our choice
of $\kappa$ above ensures  $\kappa\ge 8k$ as required in (\ref{prefabconvex}).

 Shrinking the cusps gives a submanifold $\Ccal\subset\Ccal^{\kappa}$ such that $\Ccal^{\kappa}=Th_{\kappa}(\Ccal)$.
 This gives a prefabricated manifold $Z=\Qcal_1\cup\Qcal_2\cup\Ccal$ contained in $Z^{\kappa}$.
  Then
  (\ref{prefabconvex}) implies
  $\prefab$ has a convex thickening. We remark that $Z$ and $Z^{\kappa}$ might not be connected, however any component
  will do.
 \end{proof}

\section{\label{compare} Comparison with the proof of Masters and Zhang}

The proof in \cite{MZ2} follows the same general outline. This paper is a result of our attempt
to understand their proof.
They take two (possibly not connected) QF surfaces with
boundary and glue together certain finite covers and add covers of cusps. One difference is they produce
covers so that the degree of the cover of each component surface is the same. We do not do this, but instead use
the condition of {\em simple combinatorics}. This approach avoids certain combinatorial problems concerning the compatibility
of cyclic orderings of intersection points between two surfaces as one traces around different boundary components
of these surfaces. In the approach of Masters and Zhang there
 is a big distinction depending on whether or not $M$ has only one cusp. In
our approach the number of cusps of $M$ plays no role.

We also make use of results from \cite{BC1}, and in particular the convex combination theorem. Masters and Zhang
prove and apply a special case of a version of this result. In \cite{MZ1} they introduced a 
refined version of subgroup separability for a surface with boundary. We found
a new proof  \cite{BC2} of a slight generalization of this theorem, and this result is used heavily in this paper. 
Our proof of the main theorem relies on a study of
coverings of surfaces containing certain immersed surfaces, and in particular the {\em spider theorem} (\ref{spidertheorem}). We wonder if this result about surfaces might find other applications.

%%% End main body of article

 \small
\bibliographystyle{gtart}

%\bibliography{refs} 

%\bibliographystyle{abbrv} 

\end{document}